\documentclass[dvips,preprint]{imsart}

\RequirePackage[OT1]{fontenc}
\RequirePackage{amsthm,amsmath}
\RequirePackage[colorlinks,citecolor=blue,urlcolor=blue]{hyperref}
\RequirePackage{hypernat}

\usepackage[english]{babel}
\usepackage[normalem]{ulem}
\usepackage{graphicx}	
\usepackage{amsfonts}
\usepackage{amssymb}    
\usepackage{float}
\usepackage{vmargin}

\def\var{\mbox{Var}}

\newfloat{Table}{Hptb}{lot} 

\def\1{{\mathbf 1}}

\newtheorem{thrm}{Theorem}[section]
\newtheorem{prte}[thrm]{Proposition}
\newtheorem{lemma}[thrm]{Lemma}

\newtheorem{defi}{Definition}[section]
\newtheorem{remark}{Remark}[section]
\newtheorem{ex}{Example}[section]

\newcommand{\pen}{\mathrm{pen}}

\newcommand{\supp}{\mathrm{supp}}

\numberwithin{equation}{section}

\setmarginsrb{3cm}{1.5cm}{3cm}{3.5cm}{1cm}{1cm}{1cm}{1cm}
\begin{document}

\begin{frontmatter}
\title{Minimax risks for sparse regressions:~\\ Ultra-high dimensional
phenomenons}
\runtitle{Ultra-high dimensional regression}

\begin{aug}
\author{\fnms{Nicolas}
\snm{Verzelen}\thanksref{t11}\ead[label=e11]{
nicolas.verzelen@supagro.inra.fr}}

\runauthor{Verzelen}
\thankstext{t11}{
 INRA, UMR  729 MISTEA,
F-34060 Montpellier, FRANCE.  \printead{e11}}

\end{aug}

\begin{abstract}
Consider the standard  Gaussian linear regression model ${\bf Y}={\bf
X}\theta_0+\boldsymbol{\epsilon}$, where ${\bf Y}\in\mathbb{R}^n$ is a response
vector and ${\bf X}\in\mathbb{R}^{n\times p}$ is a design matrix.
Numerous work have been devoted to building efficient estimators of $\theta_0$
when $p$ is much  larger than $n$. In such a situation, a classical approach
amounts
to assume that $\theta_0$ is  approximately sparse. This paper
studies the minimax risks of estimation and testing over classes of $k$-sparse
vectors
$\theta_0$. These bounds shed light on the limitations due to high-dimensionality.
The results encompass the problem of prediction (estimation of ${\bf
X}\theta_0$), the inverse problem (estimation of $\theta_0$) and linear testing
(testing ${\bf X}\theta_0=0$). Interestingly, an elbow
effect occurs when the number of variables $k\log(p/k)$ becomes large compared
to $n$. Indeed, the minimax
risks and hypothesis separation distances blow up in this ultra-high dimensional
setting. We also prove that even dimension reduction techniques  cannot provide
satisfying results in an ultra-high dimensional setting.  Moreover, we  compute
the minimax risks when the variance of the noise is unknown.
The knowledge of this variance is shown to play a significant role in the optimal rates of
estimation and testing. All these minimax bounds provide a characterization of statistical problems that are so difficult so that no procedure can provide satisfying results.

\end{abstract}




\end{frontmatter}

\section{Introduction}\label{section_introduction}

In many important statistical applications, including remote sensing,
functional MRI and gene expressions studies the number $p$ of parameters is
much larger than the number $n$ of observations. An active line of research
aims at developing computationally fast procedures that also achieve the
best possible statistical performances in this ``$p$ larger than $n$" setting. A typical example is the study of $l_1$-based
penalization methods for the estimation of linear regression models. However, if $p$ is really too large compared to $n$, 
all these new procedures fail to achieve a good estimation.

Thus, there is a need to understand the intrinsic limitations of a statistical problem: what is the best
rate of estimation or testing achievable by a procedure? Is it possible to
design good procedures for arbitrarily large $p$ or are there theoretical
limitations when $p$ becomes  ``too large"? These limitations tell us what kind of data analysis problems are too complex so that no
statistical procedure is able to provide reasonable results.  
Furthermore, the knowledge of such limitations may drive the research towards areas where computationally
efficient procedures are shown to be suboptimal.

\subsection{Linear regression and statistical problems}

We observe a response vector ${\bf Y}\in \mathbb{R}^n$ and
a real design matrix ${\bf X}$ of size $n\times p$. Consider the linear
regression model
\begin{eqnarray}\label{modele_generale}
 {\bf Y} = {\bf X}\theta_0 + \boldsymbol{\epsilon} \ ,
\end{eqnarray}
where the vector $\theta_0$ of size $p$ is unknown and the random vector
$\boldsymbol{\epsilon}$ follows a centered normal distribution 
$\mathcal{N}(0_n,\sigma^2 I_n)$. Here, $0_n$ stands for the null vector of
size $n$ and $I_n$ for the identity matrix of size $n$.\\

In some cases, the design ${\bf X}$ is considered as {\it fixed} either
because it has been previously chosen or because we work conditionally to
the design. In other cases, 
the rows of the design matrix ${\bf X}$ correspond
to a $n$-sample of
a random vector $X$ of size $p$. The design ${\bf X}$ is then said to be
{\it random}. A specific class of random
design is made of Gaussian designs where $X$ follows a centered normal
distribution
$\mathcal{N}(0_p,\Sigma)$. The analysis of fixed
and Gaussian designs share many common points. In
this work, we shall enhance the similarities and the differences between
both settings.\\

There are various statistical problems arising in the linear regression
model~(\ref{modele_generale}). Let us list the most classical issues:

\noindent
$({\bf P_1}):$ {\bf Linear hypothesis testing}. In general, the aim is to test
whether
$\theta_0$ belongs to a linear subspace of $\mathbb{R}^p$. Here, we focus on
testing the null hypothesis ${\bf H_0}$: $\{\theta_0=0_p\}$. In Gaussian design, this
is equivalent to testing whether ${\bf Y}$ is independent from ${\bf X}$.

\noindent
$({\bf P_2}):$ {\bf Prediction}. We focus on predicting the expectation
$\mathbb{E}[{\bf
Y}]$ in fixed design and the conditional expectation $\mathbb{E}[{\bf Y}|{\bf X}]$ in
Gaussian design. 

\noindent
$({\bf P_3}):$ {\bf Inverse problem}. The primary interest lies in estimating
$\theta_0$
itself and the corresponding loss function is
 $\|\widehat{\theta}-\theta_0\|_p^2$, where $\|.\|_p$ is the $l_2$ norm in
$\mathbb{R}^p$.

\noindent
$({\bf P_4}):$ {\bf Support estimation}  aims at recovering the support of
$\theta_0$, that is the set of indices corresponding to non-zero coefficients. The
easier problem of
{\bf dimension
reduction} amounts to estimate a set $\widehat{M}\subset \{1,\ldots p\}$
of ``reasonable" size that contains the support of $\theta_0$ with high
probability.

Much work have been devoted to these statistical questions in the so-called high-dimensional setting,
where the number of covariates $p$ is possibly much larger than $n$. A classical
approach to perform a statistical analysis in this setting is to assume that $\theta_0$ is
{\it sparse}, in the sense that most of the components of $\theta_0$ are equal to $0$.
For the problem of prediction (${\bf P_2}$),
procedures based on complexity penalization are
proved to provide good risk bounds  for known variance~\cite{BM01} and unknown
variance~\cite{BGH09} but are computationally inefficient. In contrast,
convex  penalization methods such as the Lasso or the Dantzig selector are
fast to compute, but only provide good performances  under restrictive
assumptions on
the design ${\bf X}$ (e.g. \cite{bickeltsy08,CandesPlan07,zou05}). Exponential weighted
aggregation methods~\cite{DT08,tsyrig10} are another example of fast and
efficient methods. The $l_1$ penalization methods have also
been analyzed for the inverse problem (${\bf P_3}$)~\cite{bickeltsy08} and
for support estimation (${\bf P_4}$)~\cite{lounici08,zhao06}. Dimension
reduction methods are often studied
in more general
settings than linear regression~\cite{cook02,fuku09}. In the linear regression
model, the SIS method~\cite{fan_ultra} based on the
correlation between the response and the covariates allows to perform dimension reduction. The problem of
high-dimensional hypothesis  testing (${\bf
P_1}$) has  so far attracted  less attention.
Some testing
procedures are discussed in \cite{baraud03,arias} for fixed design and
in \cite{villers1,Ingster10}  for
Gaussian design.

\subsection{Sparsity and ultra-high dimensionality}

 Given a positive
integer $k$, we say that the
vector $\theta_0$ is $k$-sparse if $\theta_0$ contains at most $k$ non-zero
components. We call $k$ the sparsity parameter. In this paper, we are interested
in the setting $k< n< p$.
We note $\Theta[k,p]$ the set of $k$-sparse vectors in
$\mathbb{R}^p$.\\

In linear regression, most of the results about classical procedures require
that the triplet $(k,n,p)$ satisfies $k[1+\log(p/k)]<n$.
When $k$ is ``small", this corresponds to assuming that $p$ is subexponential
with respect to $n$. The analysis of the
Lasso in prediction, inverse problems~\cite{bickeltsy08},
and support estimation~\cite{MB06} entail such assumptions. 
In dimension reduction, the SIS method~\cite{fan_ultra} also requires this
assumption. If the multiple
testing procedure of~\cite{baraud03} can be analyzed for
$k[1+\log(p/k)]$ larger than $n$, it  exhibits a much slower rate of testing in
this case. In noiseless problems ($\sigma=0$), compressed sensing methods~\cite{donoho_compressed} fail when $k[1+\log(p/k)]$ is large compared to $n$ (see~\cite{dota} for numerical illustrations). In the sequel, we say that the problem is
{\it ultra-high dimensional}\footnote{In some papers, the expression ultra-high dimensional has been used to characterize problems such that $\log(p)=O(n^{\beta})$ with $\beta<1$. We argue in this paper that that as soon as $k\log(p)/n$ goes to $0$, the case  $\log(p)=O(n^{\beta})$ is not intrinsically more difficult than conditions such as $p=O(n^{\delta})$ with $\delta>0$.} when $k[1+\log(p/k)]$ is large compared to $n$. Observe that ultra-high dimensionality does not necessarily
imply that $p$ is exponential with respect to $n$. As an example, taking
$p=n^3$ and $k=n/\log\log(n)$ asymptotically yields an ultra-high
dimensional problem.

Why should we care about ultra-high dimensional problem?
In this setting, there are so many variables that statistical questions such as the estimation of $\theta_0$ (${\bf P_3}$) or its support (${\bf P_4}$) are likely to be  difficult. 
Nevertheless, if the signal over noise ratio is large, do there exist estimators that perform relatively well? The answer is no. We prove in this paper that a phase transition phenomenon occurs in an ultra-high dimensional setting and that most of the estimation and testing problems become hopeless.
This phase transition phenomenon implies that some statistical problems that are tackled in postgenomic of functional MRI cannot actually be addressed properly.

\begin{ex}[{\bf Motivating example}]
In some gene network
inference problems (e.g.~\cite{chu}), the number $p$ of genes can be as
large as 5000 while the number $n$ of microarray experiments is only of
order
50. Let us consider a gene $A$. We note $\mathcal{G}_A$ the set of genes that interact with the gene $A$ and $k$ stands for the cardinality of $\mathcal{G}_A$.
How large can be  $k$  so that it is still ``reasonable" to estimate $\mathcal{G}_A$ from the microarray experiments? In statistical
terms, inferring the set of genes interacting with $A$ amounts to estimate the
support of a vector $\theta_0$ in a linear regression model (see e.g. \cite{MB06}). Our
answer is that if $k$ is larger than $4$, then the problem of network
estimation becomes extremely difficult. We will come back to
this example and explain this answer in Section \ref{section_simulations}. 
\end{ex}

\subsection{Minimax risks}

A classical way to assess the performance of an estimator $\widehat{\theta}$ is
to consider its maximal risk over a class $\Theta\subset \mathbb{R}^p$. This is
the minimax point of view. For the time being, we only define the notions of
minimaxity for 
estimation  problems (${\bf P_2}$ and ${\bf P_3}$).
Their counterpart in the case of testing (${\bf P_1}$) and dimension reduction
(${\bf P_4}$) will be introduced in subsequent
sections.
Given a loss function $l(.,.)$ and estimator $\widehat{\theta}$, the maximal
risk of $\widehat{\theta}$ over $\Theta[k,p]$ for a design ${\bf X}$ (or a
covariance $\Sigma$ in the Gaussian design case) and a
variance $\sigma^2$ is defined by
$\sup_{\theta_0\in\Theta[k,p]}\mathbb{E}_{\theta_0,\sigma}
 [ l(\widehat{\theta},\theta_0)]$. Taking the infimum of the maximal risk over all
possible estimators $\widehat{\theta}$, we obtain the
{\it minimax risk} 
$$ \inf_{\widehat{\theta}}\sup_{\theta_0\in\Theta[k,p]}\mathbb{E}_{\theta_0,\sigma}
 [ l(\widehat{\theta},\theta_0)]\ .$$ We say that an estimator
$\widehat{\theta}$
is minimax if its maximal risk over $\Theta[k,p]$ is close to the minimax risk.
\\

In practice, we do not know the number $k$ of non-zero components of $\theta_0$
and
we seldom know the variance $\sigma^2$ of the error. If an estimator $\widehat{\theta}$ does not
require the knowledge of $k$ and nearly achieves the minimax risk over
$\Theta[k,p]$ for a range of $k$, we say that $\widehat{\theta}$ is adaptive to
the sparsity. Similarly,  an estimator $\widehat{\theta}$ 
is adaptive to the variance $\sigma^2$, if it does not require the
knowledge of $\sigma^2$ and nearly achieves the minimax risk for all 
$\sigma^2>0$. When possible, the main challenge is to build adaptive procedures.
In some statistical problems considered here, adaptation is in fact impossible and
there is an
unavoidable loss when the variance or the sparsity parameter is unknown. In such
situations, it is interesting to quantify this unavoidable loss.

\subsection{Our contribution and related work}

In the specific case of the Gaussian sequence model, where $n=p$ and ${\bf
X}=I_n$, the minimax risks over $k$-sparse vectors  have been studied for a long
time. Donoho and Johnstone~\cite{donoho94,johnstone94} have provided the asymptotic minimax
risks of prediction $({\bf P_2})$. Baraud~\cite{baraudminimax} has studied the
optimal rate
of testing from a non-asymptotic point of view while Ingster~\cite{ingster99,ingster02,ingster02b} has provided the
asymptotic optimal rate of testing with
exact constants.

Recently, some high-dimensional problems have been studied from a minimax point of view. 
Wainwright~\cite{wain_minimax2,wain_minimax} provides minimax lower
bounds for the problem of support estimation (${\bf P_4}$).
Raskutti et al.~\cite{raskwain09}  and Rigollet and Tsybakov~\cite{tsyrig10} 
have provided  minimax upper bounds and lower
bounds for  $({\bf P_2})$ and $({\bf P_3})$  over $l_q$ balls  for
general fixed designs ${\bf X}$ when the variance $\sigma^2$ is known (see also Ye and Zhang~\cite{zhangminimax} and Abramovich and Grinshtein~\cite{abramovich10}).
Arias-Castro et al.~\cite{arias} and Ingster et al.~\cite{Ingster10} have computed the asymptotic minimax detection boundaries for the testing problem $({\bf P_1})$ for some specific designs. However, their study only encompasses reasonable dimensional problems ($p$ grows polynomially with $n$).
Some minimax lower bounds have also been stated for testing (${\bf P_1}$) and prediction (${\bf P_2}$) problems
with Gaussian design~\cite{verzelen_regression,villers1}. All the aforementioned results do not cover the ultra-high dimensional case and do not tackle the 
problem of adaptation to both $k$ and $\sigma$.

This paper provides minimax lower bounds and upper bounds
for the problems (${\bf P_1}$), (${\bf P_2}$), (${\bf P_3}$) when the regression
vector $\theta_0$ is $k$-sparse for fixed and random designs, known and unknown variance, known and unknown sparsities. The lower and upper bounds match up to possible differences in the logarithmic terms. The main discoveries  are the following:
\begin{enumerate}
\item {\bf Phase transition in an ultra-high dimensional setting}. Contrary to previous work, our results
cover both the high-dimensional and ultra-high dimensional setting. We establish that for 
each of the problems (${\bf P_1}$), (${\bf P_2}$) and (${\bf P_3}$), an elbow
effect occurs when $k\log(p/k)$ becomes large compared to $n$. Let us emphasize the difference between the high-dimensional and the ultra-high dimensional regimes for two problems: prediction (${\bf P_2}$) and support estimation (${\bf P_4}$).\\

{\it Prediction with random design.} In the (non-ultra) high-dimensional setting, the minimax risk of prediction for a random design regression is of order $\sigma^2 k\log(p/k)/n$ (see Section \ref{section_main_result}). Thus, the effect of the sparsity $k$ is linear and the effect of the number of variables $p$  is logarithmic. In an ultra-high dimensional setting,  that is when $k\log(p/k)/n$ is large, we establish that an elbow effect occurs in the minimax risk. In this setting, the minimax risk becomes of order $\sigma^2\exp[Ck\{1+\log(p/k)\}/n]$, where $C$ is a positive constant : it grows exponentially fast with $k$ and polynomially with $p$ (see the red curve in Figure \ref{figure02}). If it was expected that the minimax risk cannot be small for such problems, we prove here that the minimax risk is in fact exponentially larger than the usual $k\log(p/k)/n$ term.\\

{\it Support estimation.} In a non-ultra high dimensional setting it is known~\cite{wain_minimax} that under some assumptions on the design ${\bf X}$ (e.g. each component of ${\bf X}$ is drawn from iid. standard normal distribution) the support of a $k$-sparse vector $\theta_0$ is recoverable with high probability if
\begin{equation}\label{condition_support_gd}
\forall i\in\mathrm{supp}(\theta_0)\, ,\quad   |(\theta_0)_i| \geq C\sqrt{\log(p)/n}\sigma\ , 
\end{equation}
where $C$ is a numerical constant. In an ultra-high dimensional setting, even if 
\begin{equation}\label{condition_support_tgd}
\forall i\in\mathrm{supp}(\theta_0)\, ,\quad   |(\theta_0)_i| = \exp[Ck\{1+\log(p/k)\}/n]/\sqrt{k}\sigma \ ,   
\end{equation}
 it is not possible to estimate the support of $\theta_0$ with high probability. Observe that the condition (\ref{condition_support_tgd}) is much stronger than (\ref{condition_support_gd}). In fact, it is not even possible to reduce drastically the dimension of the problem without forgetting relevant variables with positive probability.
More precisely, for any dimension reduction procedure that selects a subset of variables $\widehat{M}\subset \{1,\ldots p\}$ of size $p^{\delta}$ with some $0<\delta<1$  (described in  Proposition \ref{prte_minoratoin_subset_estimation}), we have $\mathrm{supp}(\theta_0)\nsubseteq \widehat{M}$ with probability away from zero (see  Proposition \ref{prte_minoratoin_subset_estimation}). Thus, it is almost hopeless to have a reliable estimation of the support of $\theta_0$ even if $\|\theta_0\|_p^2/\sigma^2$ is large.
This impossibility of dimension reduction for ultra-high dimensional problems is numerically
illustrated in Section \ref{section_simulations}.

\item {\bf Adaptation to the sparsity $k$ and to the variance $\sigma^2$}. Most
theoretical results for the problems (${\bf P_1}$) and (${\bf
P_2}$) require that the variance $\sigma^2$ is known. Here, we establish these
minimax bounds for both known and unknown variance and  known and unknown
sparsity. The knowledge of the variance is proved to play a fundamental role
for the testing problem (${\bf P_1}$) when $k[1+\log(p/k)]$ is large
compared to $\sqrt{n}$. The knowledge of $\sigma^2$ is also proved to be crucial
for (${\bf P_2}$) in an ultra-high dimensional setting. Thus, specific work is needed to develop  fast and efficient procedures that do not require the knowledge of the variance. Furthermore, variance  estimation is extremely difficult in an ultra-high dimensional setting.

\item {\bf Effect of the design}. Lastly, the minimax bounds of $({\bf P_1})$, $({\bf P_2})$ and $({\bf P_3})$
are established for fixed and Gaussian designs. Except for the problem of
prediction $({\bf P_2})$, the minimax risks are shown to be of the same nature for both forms
of the design. Furthermore, we investigate the dependency of the minimax risks
on
the design ${\bf X}$ (resp. $\Sigma$) in Sections \ref{section_linear_testing}-\ref{section_inverse}.
\end{enumerate}

 The minimax bounds stated in this paper are non
asymptotic. While some upper bounds are consequences of recent results in the literature, most of the effort is spent here to derive the lower bound.
These bounds rely on Fano's and Le Cam's
methods~\cite{yu} and on geometric considerations. In each case, near optimal
procedures are exhibited.

\subsection{Organization of the paper}

In Section \ref{section_main_result}, we summarize the minimax bounds for
specific designs called ``worst-case" and ``best-case" designs in order to
emphasize the effects of dimensionality. The
general results are stated in Section \ref{section_linear_testing} for the tests
and Section
\ref{section_prediction} for the problem of prediction. The problems
of inverse
estimation, support estimation, and dimension
reduction are studied in Section \ref{section_inverse}. In Section
\ref{section_simulations}, we address
the following practical question: For exactly what range of $(k,p,n)$ should we 
consider a statistical problem as ultra-high dimensional? A small simulation
study illustrates this answer. Section
\ref{section_conclusion} contains the final discussion and side results about variance estimation. Section
\ref{section_proof_lowerbound} is devoted to the proof of the mains minimax lower
bounds. Specific statistical procedures allow to establish the minimax upper bounds. Most of these procedures are used as theoretical tools but should not be applied in a high dimensional setting because they are computationally inefficient.  In order to clarify the statements of the results in Sections \ref{section_linear_testing}--\ref{section_inverse}, we postpone the definition of these procedures to Section \ref{section_proof_upperbound}.
The remaining proofs are described in a technical appendix~\cite{technical}.

\section{Notations and preliminaries}\label{section_notation}

We respectively note $\|.\|_n$ and $\|.\|_p$  the $l_2$  norms in
$\mathbb{R}^n$ and $\mathbb{R}^p$, while $\langle .\rangle_n$ refers to the
inner product in $\mathbb{R}^n$. For any $\theta_0\in\mathbb{R}^p$ and
$\sigma>0$, $\mathbb{P}_{\theta_0,\sigma}$ and $\mathbb{E}_{\theta_0,\sigma}$
refer to the joint distribution of $({\bf Y},{\bf X})$. When there is no risk of 
 confusion, we simply write $\mathbb{P}$ and $\mathbb{E}$. All references with a capital letter such as Section A or Eq.(A.3) refer to the technical Appendix~\cite{technical}.

In the sequel, we note
$\mathrm{supp}(\theta_0)$ the support of $\theta_0$. For any $1\leq k\leq p$,
$\mathcal{M}(k,p)$ stands for the collections of all
subsets of $\{1,\ldots, p\}$ with cardinality  $k$. Given $i\in\{1,\ldots,p\}$, we note ${\bf X}_i$ the vector of size $n$ corresponding to $i$-th column of ${\bf X}$.
For $m\subset\{1,\ldots,p\}$,  ${\bf X}_m$ stands for the $n\times |m|$ submatrix of ${\bf X}$ that contains the columns ${\bf X}_i$, $i\in m$.
In what follows, we note ${\bf X}^T$ the transposed matrix of ${\bf X}$. \\

\noindent 
{\bf  Gaussian design and conditional distribution}. When the design is said to be ``Gaussian", the $n$
rows of ${\bf X}$ are $n$ independent samples of a random row vector $X$ such that $X^T\sim \mathcal{N}(0_p,\Sigma)$. Thus, $({\bf
Y},{\bf X})$ if a $n$-sample of the
random vector $(Y,X^T)\in\mathbb{R}^{p+1}$, where $Y$ is defined by
\begin{eqnarray}\label{modele_gaussien}
 Y= X\theta_0+ \epsilon\ ,
\end{eqnarray}
 where $\epsilon\sim\mathcal{N}(0,\sigma^2)$. The linear regression model with
{\it Gaussian} design is relevant to understand the conditional distribution of
a Gaussian variable $Y$ conditionally to a Gaussian vector since
$\mathbb{E}[Y|X]=X\theta_0$ and $\var(Y|X)=\sigma^2$. This is why we shall often
refer to $\sigma^2$ as the conditional variance of $Y$ when considering
Gaussian design. This model is also closely
connected to the estimation of  Gaussian graphical
models~\cite{MB06,villers1}.\\

As explained later, the minimax risk
over $\Theta[k,p]$ strongly depends on the design ${\bf X}$. This is why we
introduce some relevant quantities on ${\bf X}$.
\begin{defi}\label{definition_propriete_RIP}
Consider some integer $k>0$ and some design ${\bf X}$.
\begin{eqnarray}
\varPhi_{k,+}({\bf X}) :=
\sup_{\theta\in \Theta[k,p]\setminus\{0_p\}}\frac{\|{\bf
X}\theta\|_n^2}{\|\theta\|_p^2}\hspace{0.5cm}\text{ and }\hspace{0.5cm}
\varPhi_{k,-}({\bf X}) :=\inf_{\theta\in\Theta[k,p]\setminus\{0_p\}}\frac{\|{\bf
X}\theta\|_n^2}{\|\theta\|_p^2}\ .
\end{eqnarray}
In fact, $\varPhi_{k,+}({\bf X})$ and $\varPhi_{k,-}({\bf X})$ respectively
correspond
to the largest and the smallest restricted eigenvalue of order $k$ of ${\bf
X}^T{\bf X}$.
\end{defi}
Given a symmetric real square matrix $A$, $\varphi_{\max}(A)$ stands for the
largest eigenvalue of $A$.
Finally, $C$, $C_1$,$\ldots$ denote positive universal constants that may vary
from line to line. The notation $C(.)$ specifies
the dependency on some quantities.

In the propositions, the constants involved
in the assumptions are not always expressly specified. For instance, sentences
of the
form ``Assume that $n\geq C$. Then, $\ldots$" mean that  ``There exists an
universal $C>0$ such that if $n\geq C$, then $\ldots$".

\section{Main results}\label{section_main_result}

The exact bounds are stated in Section
\ref{section_linear_testing}--\ref{section_inverse}. In order to explain these
results, 
we now summarize the main minimax bounds by focusing  on the role of $(k,n,p)$
 rather than on the dependency on the design ${\bf X}$. In
order to keep the notations short, we do not provide in this section the minimal
assumptions of the results. Let us simply mention that all of them  are
valid if the sparsity $k$ satisfies $k\leq (p^{1/3})\wedge (n/5)$ and that $p\geq
n\geq C$ where $C$
a positive numerical constant.

\subsection{Prediction}\label{section_main_prediction}

\subsubsection{Definitions}
First, the results are described for the problem of prediction (${\bf P_2}$)
since the problem of minimax estimation is more classical in this setting.
Different prediction loss functions are used for fixed and Gaussian
designs. When the design is considered as fixed, we study the loss
$\|{\bf
X}(\theta_1-\theta_2)\|_n^2/(n\sigma^2)$. For  Gaussian design, we consider the
integrated prediction  loss function:
\begin{eqnarray}\label{definition_perte_l}
\|\sqrt{\Sigma}(\theta_1-\theta_2)\|_p^2/\sigma^2
=\mathbb{E}\left[\{X(\theta_1-\theta_2)\}^2\right]/\sigma^2\ .
\end{eqnarray}
Given a design ${\bf X}$, the minimax risk of prediction over $\Theta[k,p]$
with respect to ${\bf X}$ is 
\begin{eqnarray}\label{definition_minimax_prediction_fixe}
\mathcal{R}_F[k,{\bf
X}]=\inf_{\widehat{\theta}}\sup_{\theta_0\in\Theta[k,p]}\mathbb{E}_{\theta_0,\sigma}
[ \| { \bf
X}(\widehat{\theta}-\theta_0)\|_n^2/(n\sigma^2)]\ . 
\end{eqnarray}
For a Gaussian design with
covariance $\Sigma$, we study the quantity
\begin{eqnarray}\label{definition_minimax_prediction_random}
\mathcal{R}_R[k,\Sigma]:=\inf_{\widehat{\theta}}\sup_{\theta_0\in\Theta[k,p]}
\mathbb { E }_{\theta_0,\sigma } [
\| \sqrt{\Sigma}(\widehat{\theta}-\theta_0)\|_p^2/\sigma^2]\ . 
\end{eqnarray}
These minimax risks of prediction do not only depend on
$(k,n,p)$ but also on the design ${\bf X}$ (or on the covariance $\Sigma$). The
computation of the exact dependency of the minimax risks on ${\bf X}$ or
$\Sigma$ is a challenging question. 
To simplify the presentation in this section, we only describe the minimax
prediction risks for worst-case designs defined by
\begin{equation}\label{definition_minimax_prediction_global}
\mathcal{R}_{F}[k]:=\sup_{\bf X}\ \mathcal{R}_F[k,{\bf
X}],\quad  \mathcal{R}_{R}[k] :=\sup_{\Sigma}\mathcal{R}_R[k,\Sigma]\ ,
 \end{equation}
the supremum being taken over all designs ${\bf X}$ of size $n\times p$ (resp.
all covariance matrices $\Sigma$). 
The quantity $\mathcal{R}_{F}[k]$  corresponds to
the smallest risk achievable {\it uniformly} over $\Theta[k,p]$ {\it and} all
designs ${\bf X}$. It is shown in Section \ref{section_prediction} that the quantity $\mathcal{R}_R[k]$ is achieved (up to constants) for a covariance $\Sigma=I_p$ while the quantity $\mathcal{R}_F[k]$ is achieved with high probability for designs ${\bf X}$ that are realizations of the standard Gaussian design (all the components of ${\bf X}$ are drawn independently from a standard normal distribution). This corresponds to designs used in compressed sensing~\cite{donoho_compressed}. In fact, the maximal risks $\mathcal{R}_F[k]$ and $\mathcal{R}_R[k]$ for the prediction problem correspond to typical situations where the designs is well-balanced, that is as close as possible to orthogonality.

\subsubsection{Results}

In the
sequel, we say that $\mathcal{R}_{F}[k]$ is \emph{of order} $f(k,p,n,C)$, where $C$ is positive constant when
there exist two positive universal constants $C_1$ and
$C_2$ such that
$$f(k,p,n,C_1)\leq \mathcal{R}_{F}[k]\leq
f(k,p,n,C_2)\ .$$

These minimax risks are computed in Section \ref{section_prediction} and are
gathered in Table \ref{tab2}. They are also depicted on Figure \ref{figure02}.

\begin{figure}[hptb]
\begin{center}
{\includegraphics[width=6cm,angle=0]{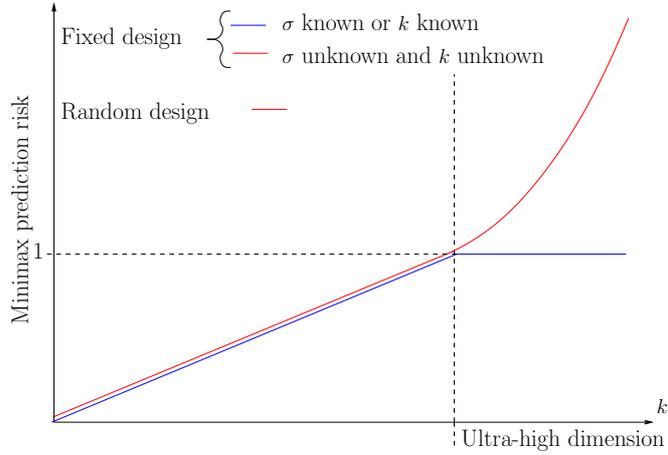}}
\caption{Minimax prediction risk (${\bf P_2}$) over $\Theta[k,p]$ as a function of $k$ for  fixed and random design and known and unknown variance.  The corresponding bounds are stated in Section \ref{section_prediction}.}\label{figure02} 
\end{center}
\end{figure}

\begin{Table}[h] 

\caption{Orders of magnitude of the  minimax risks of prediction $\mathcal{R}_{F}[k]$ and
$\mathcal{R}_{R}[k]$ over $\Theta[k,p]$. \label{tab2} }

\begin{center}

\begin{tabular}{c|c}
  {\normalsize{\bf Fixed} Design}: $\mathcal{R}_{F}[k]$ & {\normalsize{\bf Gaussian} Design}: $\mathcal{R}_{R}[k]$
 \\
\hline \\
 \hspace{0.7cm} $C\left[\frac{\displaystyle
k}{\displaystyle n}\log\left(\frac{\displaystyle p}{\displaystyle k}\right)\right]\wedge 1$\hspace{0.8cm} & \hspace{0.5cm} 
$C_1\frac{\displaystyle k}{\displaystyle
n}\log\left(\frac{\displaystyle p}{\displaystyle k}\right)\exp\left[C_2\frac{\displaystyle k}{\displaystyle n}\log\left(\frac{\displaystyle p}{\displaystyle k}\right)\right]$
\hspace{0.3cm}
\end{tabular}
 
\end{center}
\end{Table}

When $k\log(p/k)$ remains small compared to $n$, the minimax risk of
prediction is of the same order for fixed and Gaussian design. The
$k\log(p/k)/n$ risk is classical
and has been known for a long time in the specific case of the Gaussian sequence
model~\cite{johnstone94}. Some procedures based on complexity
penalization or aggregation (e.g. \cite{BM01}) are
proved to achieve these risks uniformly over all designs ${\bf X}$.
Computationally efficient procedures like the Lasso or the Dantzig selector
are only proved to achieve a $k\log(p)/n$ risk under assumption on the design ${\bf
X}$~\cite{bickeltsy08}. If the support of $\theta_0$ is known in advance, the
parametric risk is of order $k/n$. Thus, the price to pay for
not knowing the support of $\theta_0$ is only logarithmic in $p$.\\

In an ultra-high dimensional setting, the minimax prediction risk in fixed
designs remains smaller than one. It is the minimax risk of estimation of the
vector  $\mathbb{E}({\bf Y})$ of size $n$. This means that the sparsity index
$k$ 
does not play anymore a role in ultra-high dimension. For a Gaussian  design,
the minimax prediction risk becomes of order $C_1
(p/k)^{C_2k/n}$: it increases
exponentially fast with respect to $k$ and polynomially fast with respect to $p$.
Comparing this risk with the parametric rate $k/n$, we observe that the price to
pay for not knowing
the support of $\theta_0$ is now far higher than $\log(p)$.\\

In Section \ref{section_prediction}, we also study the adaptation to the
sparsity index $k$ and to the variance $\sigma^2$. We prove that
adaptation to $k$ and $\sigma^2$ is possible for a Gaussian design. In fixed
design, no procedure can be simultaneously adaptive to the sparsity $k$ and the
variance $\sigma^2$ (see the red curve in Figure \ref{figure02} that corresponds to fixed design, $\sigma$ and $k$ unknown).

\subsection{Testing}\label{section_resultat_test}
\subsubsection{Definitions}
Let us turn to the problem $({\bf P_1})$ of testing ${\bf H_0}$: $\{\theta_0=0_p\}$
against ${\bf H_1}$: $\{\theta_0\in \Theta[k,p]\setminus\{0_p\}\}$. We fix a level
$\alpha>0$ and  a type II error probability $\delta>0$. Minimax lower and
upper bounds for this problem are discussed in Section \ref{section_linear_testing}.

Suppose we are given a test procedure  $\Phi_{\alpha}$ of level $\alpha$ for
fixed design ${\bf X}$ and known variance $\sigma^2$. The $\delta$-separation
distance of
$\Phi_{\alpha}$ over $\Theta[k,p]$, noted $\rho_F[\Phi_{\alpha},k,{\bf X}]$ is
the minimal number $\rho$, such that
$\Phi_{\alpha}$ rejects ${\bf H_0}$ with probability larger than $1-\delta$ if 
$\|{\bf X}\theta_0\|_n/\sqrt{n}\geq
\rho\sigma$. Hence, $\rho_F[\Phi_{\alpha},k,{\bf X}]$ corresponds to
the minimal distance such that the hypotheses $\{\theta_0=0_p\}$ and $\{\theta_0\in\Theta[k,p]$, $\|{\bf
X}\theta_0\|_n^2\geq n\rho_F^2[\Phi_{\alpha},k,{\bf X}]\sigma^2\}$
are
well separated by the test $\Phi_{\alpha}$.
\begin{eqnarray*}\label{defition_distance_separation}
\rho_F[\Phi_{\alpha},k,{\bf X}]:=
\inf\left\{\rho>0,\, \inf_{\theta_0\in\Theta[k,p],\ \|{\bf X}\theta_0\|_n\geq
\sqrt{n}\rho\sigma}\mathbb{P}_{\theta_0,\sigma}[\Phi_{\alpha}=1]\geq
1-\delta\right\}\ .
\end{eqnarray*}
Although the separation distance also depends on  $\delta$, $n$, and $p$, we only write $\rho_F[\Phi_{\alpha},k,{\bf X}]$ for the sake of conciseness.
By definition, the test $\Phi_{\alpha}$ has a power larger than $1-\delta$ for $\theta_0\in\Theta[k,p]$ such that $\|{\bf X}\theta_0\|_n^2\geq \rho_F^2[\Phi_{\alpha},k,{\bf X}]$. Then, we consider
\begin{eqnarray}\label{definition_vitesse_minimax_test}
\rho_F^{*}[k,{\bf X}]:=
\inf_{\Phi_{\alpha}}\rho[\Phi_{\alpha},k,{\bf X}]\ .
\end{eqnarray}
The infimum runs over all level-$\alpha$ tests. We call this quantity the
$(\alpha,\delta)$-minimax separation distance over $\Theta[k,p]$ with design
${\bf X}$ and variance $\sigma^2$. The minimax separation distance is
a non-asymptotic counterpart of the detection boundaries
studied in the Gaussian sequence model~\cite{donoho04}.

Similarly, we define the $(\alpha,\delta)$-minimax separation distance over
$\Theta[k,p]$ with  Gaussian design by replacing the distance $\|{\bf
X}\theta_0\|_n/\sqrt{n}$ by the distance $\|\sqrt{\Sigma}\theta_0\|_p$:
\begin{equation}\label{definition_vitesse_minimax_test_gaussien}
\rho_R[\Phi_{\alpha},k,\Sigma]:= \inf\Big\{\rho>0,
\inf_{
\scriptstyle \theta_0\in\Theta[k,p],\ 
\scriptstyle \|\sqrt{\Sigma}\theta_0\|_p\geq
\rho\sigma
      }\mathbb{P}_{\theta_0,\sigma}[\Phi_{\alpha}=1]\geq
1-\delta\Big\}\ ,\, \rho_R^{*}[k,\Sigma]:=
\inf_{\Phi_{\alpha}}\rho_R[\Phi_{\alpha},k,\Sigma]\ .
\end{equation}
Various bounds on $\rho^*_F[k,{\bf X}]$, $\rho^*_R[k,\Sigma]$ are stated in Section \ref{section_linear_testing}.
In this section, we only provide the orders of magnitude  of the
minimax
separation distances in the ``worst case" designs in order to emphasize the effect of dimensionality:
\begin{eqnarray}\label{definition_vitesse_minimax_pire_des_cas}
 \rho^{*}_F[k] := \sup_{{\bf
X}}\rho_F^{*}[k,{\bf X}]\ ,\hspace{1.5cm} \rho^{*}_R[k] :=
\sup_{\Sigma}\rho_R^{*}[k,\Sigma]\ .
\end{eqnarray}
This is the smallest separation distance that can be
achieved by a procedure $\Phi_{\alpha}$ {\it uniformly} over all designs ${\bf
X}$ (resp. $\Sigma$). As for the prediction problem, it will be proved in Section \ref{section_linear_testing}, that the quantity $\rho^*_F[k]$ and $\rho^*_R[k]$ are achieved for well-balanced designs.

It is
not always possible to achieve the minimax separation distances with a
procedure $\Phi_{\alpha}$ that \emph{does not require} the knowledge of the
variance $\sigma^2$. This is why we also consider $\rho^{*}_{F,U}[k]$ and
$\rho^{*}_{R,U}[k]$ the minimax separation distance for fixed and Gaussian
design when the variance is unknown. Roughly, $\rho^{*}_{F,U}[k]$ corresponds
to the minimal distances
$\rho^2$ that allows to separate well the hypotheses $\{\theta_0=0_p\ \text{ and }\ \sigma>0\}$ and $\{\theta_0\in\Theta[k,p]\text{ and }\sigma>0\ ,\  \|{\bf
X}\theta_0\|_n^2/\sigma^2\geq n\rho^2\}$ when $\sigma$ is unknown. We shall provide
a formal definition at the beginning of Section \ref{section_linear_testing}.

\subsubsection{Results}

In Table \ref{tab1}, we provide the orders of the minimax separation
distances over $\Theta[k,p]$ for fixed and Gaussian designs, known and unknown
variance (see also Figure \ref{figure01}).

\begin{figure}[htbp]

\begin{center}
\includegraphics[width=6cm,angle=0]{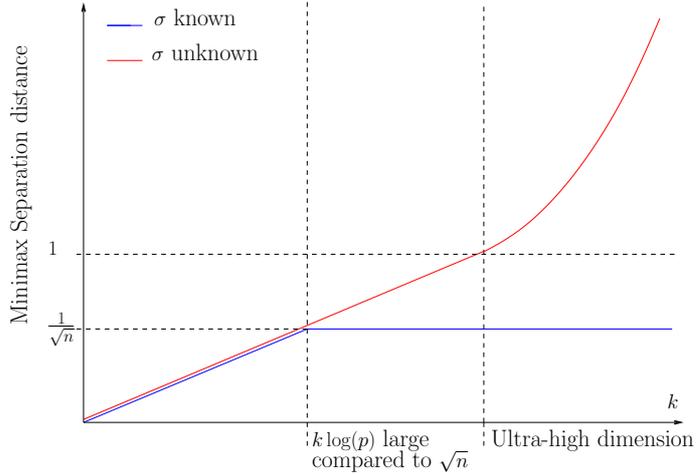}
\caption{Orders of magnitude of the minimax separation distances $(\rho^*_F[k])^2$, $(\rho^*_R[k])^2$, $(\rho^*_{F,U}[k])^2$ and $(\rho^*_{R,U}[k])^2$ over $\Theta[k,p]$ (${\bf P_1}$) for fixed and random designs and known and unknown variances. Here, $\rho^*_F[k]$ and $\rho^*_R[k]$ behave similarly while $\rho^*_{F,U}[k]$ and $\rho^*_{R,U}[k]$ behave similarly. The corresponding bounds are stated in Section \ref{section_linear_testing}.}\label{figure01} 
\end{center}
\end{figure}

\begin{Table}[h] 
\caption{Order of the minimax separation distances over $\Theta[k,p]$
for fixed and Gaussian
design, known and unknown variance:
$(\rho_F^*[k])^2$, $(\rho_R^*[k])^2$, $(\rho_{F,U}^*[k])^2$, and
$(\rho_{R,U}^*[k])^2$.\label{tab1}}

\begin{center}

\begin{tabular}{c|c}
  &{\normalsize {\bf Fixed} and {\bf Gaussian} Design}\\ 
\hline \\ {\normalsize{\bf Known} $\sigma^2$}: $(\rho^{*}_F[k])^2$  {\normalsize
and} $(\rho^{*}_R[k])^2$&
$C(\alpha,\delta)\frac{\displaystyle k\log(p)}{ \displaystyle
n}\wedge\frac{1}{\displaystyle
\sqrt{n}}$ \\
{\normalsize{\bf Unknown} $\sigma^2$}: $(\rho^{*}_{F,U}[k])^2$ {\normalsize and}
$(\rho^{*}_{R,U}[k])^2$ & 
$C(\alpha,\delta)\frac{\displaystyle
k\log(p)}{\displaystyle n}\exp\left[C_2(\alpha,\delta)\frac{\displaystyle
k\log(p)}{\displaystyle
n}\right]$
\end{tabular}
 
\end{center}
\end{Table}

In contrast to $({\bf P_2})$, the minimax separation distances are of the same
order for fixed and Gaussian design. 

\begin{enumerate}
 \item  When $k\log(p)\leq \sqrt{n}$, all the minimax separation distances are of
order $k\log(p)/n$. This quantity also corresponds to the minimax risk of
prediction $({\bf P_2})$ stated in the previous subsection. This separation
distance has already been proved in the specific case of the Gaussian sequence
model~\cite{baraudminimax,donoho04}.

\item When $k\log(p)\geq \sqrt{n}$, the minimax separation distances are different
under known and unknown variance.  If the variance
is known, the minimax separation distance over $\Theta[k,p]$ stays of order
$1/\sqrt{n}$. Here, $1/\sqrt{n}$ corresponds in fixed design to the minimax
separation distance of the hypotheses
$\{\mathbb{E}[{\bf Y}]=0_n\} $ against the general hypothesis $\{ \mathbb{E}[{\bf
Y}]\neq
0_n\}$ for known variance (see Baraud~\cite{baraudminimax}).

\item If the variance is unknown, the minimax separation distance over
$\Theta[k,p]$ is still of order  $k\log(p)/n$ if
$k\log(p)$ is small compared to $n$. In contrast, the
minimax separation distance blows up to the order $C_1p^{C_2k/n}$ in a
ultra-high dimensional setting. This blow up phenomenon has also been observed
in the previous section for the problem  of prediction~$({\bf P_2})$ in Gaussian
design. In conclusion, the knowledge of the variance is
of great importance for $k\log(p)$ larger than $\sqrt{n}$.
\end{enumerate}

\subsection{Inverse problem and support
estimation}\label{section_main_result_inverse}
\subsubsection{Definitions}
In the inverse problem (${\bf P_3}$), we are primarily interested in the
estimation of
$\theta_0$ rather than ${\bf X}\theta_0$. This is why the loss function under
study is $\|\theta_1-\theta_2\|_p^2$. Minimax lower and upper bounds for this
loss function are discussed in Section \ref{section_inverse}. For a fixed design
${\bf X}$, the minimax risk of estimation is 
\begin{equation}\label{defi_risque_minimax_inverse}
\mathcal{RI}_F[k,{\bf X}]:=
\inf_{\widehat{\theta}}\sup_{\theta_0\in\Theta[k,p]}\mathbb{E } _ { \theta_0 ,
\sigma }
 [ \|\theta_0-\widehat{\theta}\|_p^2/\sigma^2]\ .
\end{equation}

If one transforms the design ${\bf X}$ by an homothety of factor $\lambda>0$,
then this multiplies the minimax risk for the inverse problem by a factor
$1/\lambda^{2}$. For the sake of simplicity, 
 we restrict ourselves
to designs ${\bf X}$ such that each column has been normed to $\sqrt{n}$. The collection of
such designs is noted
$\mathcal{D}_{n,p}$. The supremum of the minimax
risks over the designs $\mathcal{D}_{n,p}$ is $+\infty$. Take for instance a
design where the two first columns are equal.  In this section, we only present the
infimum of the minimax risks over $\Theta[k,p]$ as ${\bf X}$ varies across
$\mathcal{D}_{n,p}$:
\begin{eqnarray*}
\mathcal{RI}_F[k] :=  \inf_{{\bf X}\in \mathcal{D}_{n,p}}\mathcal{RI}_F[k,{\bf X}]\
.
\end{eqnarray*}
The quantity $\mathcal{RI}_F[k]$ is interpreted the following way: given 
$(k,n,p)$ what is the
smallest risk we can hope if we use the best possible design? Alternatively, given $n$ observations, what is the intrinsic difficulty of estimating a $k$-sparse vector of size $p$?
We call this quantity the minimax risks for the inverse problem over
$\Theta[k,p]$.\\

In Section \ref{section_inverse}, we also study the corresponding the minimax risks of the inverse problem in the random design case. Let $\mathcal{S}_{p}$ stand for the set of covariance matrices that contain only ones on the diagonal. We respectively define  the minimax risk of estimation over $\Theta[k,p]$ for a covariance $\Sigma$ and the minimax risk of estimation over $\Theta[k,p]$ as 
\begin{equation}\label{defi_risque_minimax_inverse_random}
\mathcal{RI}_R[k,\Sigma]:=
\inf_{\widehat{\theta}}\sup_{\theta_0\in\Theta[k,p]}\mathbb{E } _ { \theta_0 ,
\sigma }
 [ \|\theta_0-\widehat{\theta}\|_p^2/\sigma^2]\quad \text{ and }\quad  \mathcal{RI}_R[k]:= \inf_{\Sigma\in\mathcal{S}_p}\mathcal{RI}_R[k,\Sigma]\ .
\end{equation}

\subsubsection{Results}

In Table \ref{tab3}, we provide the minimax risks in fixed
design for different
values of $(k,n,p)$ (see also Figure \ref{figure03}).

\begin{Table}[h]
\caption{Order of the minimax risks  $\mathcal{RI}_F[k]$ for the inverse problem  over
$\Theta[k,p]$\label{tab3} } 
\begin{center}

\begin{tabular}{c|c|c}
{\normalsize ${\bf (k,n,p)}$}  & ${\displaystyle k\log(p)\leq Cn}$ &
${\displaystyle k\log(p)\gg n\log(n)}$\\ \hline&& \\
{\normalsize Minimax risk} $\mathcal{RI}_F[k]$& $C\frac{k}{n}\log\left(\frac{\displaystyle p}{\displaystyle k}\right)$ & 
$\exp\left[C'\frac{\displaystyle k}{\displaystyle n}\log\left(\frac{\displaystyle p}{\displaystyle k}\right)\right]$.
\end{tabular}
 
\end{center}
\end{Table}

\begin{figure}[hptb]
\begin{center}
 
\includegraphics[width=6cm,angle=0]{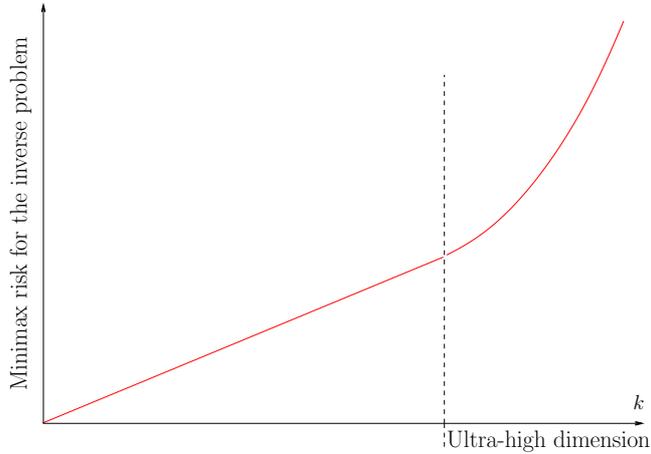}

\caption{Order of magnitude of the minimax risk $\mathcal{RI}_F[k]$ for the inverse problem (${\bf P_3}$) over $\Theta[k,p]$ as a function of $k$. The corresponding bounds are stated in Section \ref{section_inverse}.}\label{figure03}  

\end{center}

\end{figure}

If $k\log(p/k)$ remains smaller than $n$, it is possible to recover
the risk $Ck\log(p/k)$ for ``good" designs. This risk is for instance
achieved by the Dantzig
selector of Cand\`es and Tao~\cite{candes07} for nearly-orthogonal designs,
that roughly means that the restricted eigenvalues $\varPhi_{3k,+}({\bf X})$ and
$\varPhi_{3k,-}({\bf X})$ of ${\bf X}^T{\bf X}$ are close
to one. In an ultra high-dimensional setting, it is not anymore
possible to build nearly-orthogonal designs ${\bf X}$ and the minimax risk of
the inverse
problem blows up as for testing problems (${\bf P_1}$) or prediction problems
in Gaussian design (${\bf P_2}$). Moreover, adaptation to the sparsity $k$ and to the variance $\sigma^2$ is possible for the inverse problem. 
As explained  in Section \ref{section_inverse}, the quantities $\mathcal{RI}_R[k,\Sigma]$ and $\mathcal{RI}_R[k]$ behave somewhat similarly to their fixed design counterpart.\\

In Section \ref{section_inverse}, we also discuss the
consequences of the minimax bounds on the problem of support estimation (${\bf
P_4}$).  We prove that, in an ultra-high dimensional setting, it is not
possible to
estimate with high probability the support of $\theta_0$ unless the ratio
$\|\theta_0\|_p^2/\sigma^2$ is larger than $C_1(p/k)^{C_2k/n}$.
In fact, even the problems of  support estimation is almost
hopeless in an ultra-high dimensional setting.

\section{Hypothesis Testing}\label{section_linear_testing}

We start by the testing problem (${\bf P_1}$) because some minimax
lower bounds in prediction
and inverse estimation derive from testing considerations.

\subsection{Known variance}

\subsubsection{Gaussian design}\label{section_test_aleat_variance_connue}

As mentioned in the introduction, the knowledge of $\sigma^2=\var(Y|X)$ is
really unlikely in many practical applications. Nevertheless, we study this
case to
enhance the differences between known and unknown conditional variances.
Furthermore, these results turn out to be useful for analyzing the minimax
separation
distances in fixed design problems. We recall that the notions of minimax
separation distances $\rho^*_F[k,{\bf X}]$, $\rho^*_F[k]$, $\rho^*_R[k,\Sigma]$, and $\rho^*_R[k]$ have been defined in Section \ref{section_resultat_test}.

\begin{thrm}\label{thrm_minoration_minimax_variance_conditionnelle} 

Assume that $\alpha+\delta\leq 53\%$, $p\geq n^2$, and that $n\geq 8\log(2/\delta)$. For any $1\leq k\leq n$, the $(\alpha,\delta)$-minimax separation distance
(\ref{definition_vitesse_minimax_test_gaussien}) with covariance $I_p$ is lower
bounded by
\begin{eqnarray}\label{minoration_test_random_eq_known_variance2}
 (\rho_R^*[k,I_p])^2\geq
C\left[\frac{k}{n}\log\left(p\right)\wedge \frac{1}{\sqrt{n}}\right]\ .
\end{eqnarray}
 For any  $1\leq k\leq p$ and any covariance $\Sigma$, we have 
\begin{eqnarray}\label{majoration_test_random_design_variance_connue_eq}
(\rho_R^*[k,\Sigma])^2\leq 
C(\alpha,\delta)\left[\frac{k}{n}\log\left(p\right)\wedge
\frac{1}{\sqrt{n}}\right]\ .
\end{eqnarray}
Furthermore, this upper bound is simultaneously achieved for all $k$ and $\Sigma$ by a procedure $T_{\alpha}^*$ (defined in Section \ref{section_testetoile}).
\end{thrm}

\begin{remark}{\bf [Adaptation to sparsity]}
It follows from Theorem \ref{thrm_minoration_minimax_variance_conditionnelle} that adaptation to the sparsity is possible and that the optimal 
optimal separation distance is of order 
\begin{equation}\label{equation_ordre_vitesse_minimax}
 \frac{k}{n}\log\left(p\right)\wedge
\frac{1}{\sqrt{n}}\ ,
\end{equation}
for all sparsities $k$ between $1$ and $n$. 
\end{remark}

\begin{remark}{\bf [Correlated design]}
The upper bound  (\ref{majoration_test_random_design_variance_connue_eq})  is valid for any covariance matrix 
$\Sigma$. In contrast, the minimax lower bound (\ref{minoration_test_random_eq_known_variance2}) is restricted to the case $\Sigma=I_p$.  This
implies that there exists some constant $C(\alpha,\delta)$ such that ,
$$\rho_R^*[k,I_p]\geq C(\alpha,\delta) \sup_{\Sigma}\rho_R^*[k,\Sigma]= C(\alpha,\delta)\rho_R^*[k]\ .$$
In other words, the testing problem is more complex (up to constants) for an independent design than for a correlated design.
\end{remark}

\begin{remark}{\bf [Which logarithmic term in the bound: $\log(p)$ or $\log(p/k)$?]}
In the proof of Theorem \ref{thrm_minoration_minimax_variance_conditionnelle}, we derive the following bounds 
\begin{eqnarray*}
 (\rho_R^*[k,I_p])^2&\geq&
C\left[\frac{k}{n}\log\left(1+\frac{p}{k^2}\right)\wedge \frac{1}{\sqrt{n}}\right]\ ,\\
(\rho_R^*[k,\Sigma])^2 &\leq& 
C(\alpha,\delta)\left[\frac{k}{n}\log\left(\frac{ep}{k}\right)\wedge
\frac{1}{\sqrt{n}}\right]\ .
\end{eqnarray*}
 These two bounds are of order of (\ref{equation_ordre_vitesse_minimax}) as it is assumed that $p\geq n^2$. However, the dependency of the logarithmic terms on $k$ in the last bounds do not allow to provide the minimax separation distance when $p=n$ and $k$ is close to $\sqrt{n}$.
For instance, if $p=n$ and $k=\sqrt{n}/\log(n)$, the two bounds only match up to a factor $\log(n)/\log\log(n)$. The non-asymptotic minimax bounds of Baraud~\cite{baraudminimax} in the Gaussian sequence model suffer the same weakness. Up to our knowledge the dependency on $\log(k)$ of the minimax separation distances has only been captured in an asymptotic setting~\cite{arias,Ingster10} ($(k,p,n)\rightarrow \infty$).
\end{remark}

\subsubsection{Fixed design}

The separation distances are similar to the  Gaussian design case.

\begin{thrm}\label{prte_minoration_test_fixed}
Assume that $\alpha+\delta\leq 33\%$, $p\geq n^2\geq C(\alpha,\delta)$, and that $n\geq 8\log(2/\delta)$.
For any $1\leq k\leq n$, there exist some $n\times
p$ designs ${\bf X}$ such that 
\begin{eqnarray}\label{minoration_test_fixed_eq_known_variance}
 (\rho_F^*[k,{\bf X}])^2\geq 
C\left[\frac{k}{n}\log\left(p\right)\wedge \frac{1}{\sqrt{n}}\right]\ .
\end{eqnarray}
For any  $1\leq k\leq p$ and any design ${\bf X}$, we have 
\begin{eqnarray}\label{majoration_test_fixed_design_variance_connue_eq}
(\rho_F^*[k,{\bf X}])^2
\leq
C(\alpha,\delta)\left[\frac{k}{n}\log\left(p\right)\wedge
\frac{1}{\sqrt{n}}\right]\ .
\end{eqnarray}
Furthermore, this upper bound is simultaneously achieved for all $k$ and ${\bf X}$ by a procedure $T_{\alpha}^*$ (defined in Section \ref{section_testetoile}).
\end{thrm}
As for the random design case, we conclude that adaptation to the sparsity is possible and that $(\rho_F^*[k])^2$ is of order $\frac{k}{n}\log\left(p\right)\wedge
\frac{1}{\sqrt{n}}$.
In fact, the proof shows that, with large probability, designs ${\bf X}$ whose components are independently sampled from a standard normal variable satisfy (\ref{minoration_test_fixed_eq_known_variance}).

Arias-Castro et al.~\cite{arias} and Ingster et al.~\cite{Ingster10} have recently provided the asymptotic minimax separation distance with exact constant for known variance when the design satisfies very specific conditions. Theorem \ref{prte_minoration_test_fixed} provides the non-asymptotic counterpart of their result, but  the constants in (\ref{minoration_test_fixed_eq_known_variance}) and (\ref{majoration_test_fixed_design_variance_connue_eq}) are not optimal.

\subsection{Unknown variance}

\subsubsection{Preliminaries}

We now turn to the study of the minimax separation distances when the variance
$\sigma^2$ is unknown. In Section \ref{section_resultat_test}, we have
introduced the notions of
$\delta$-separation distances and $(\alpha,\delta)$-minimax separation
distances when the variance $\sigma^2$. We now define their counterpart for an
unknown variance $\sigma^2$.

Let us consider a test $\Phi_{\alpha}$ of the hypothesis ${\bf H_0}$ for the
linear regression model with fixed design ${\bf X}$. We say that $\Phi_{\alpha}$
has a level $\alpha$ under unknown variance if
$$\sup_{\sigma>0}\mathbb{P}_{0_p,\sigma}[\Phi_{\alpha}({\bf Y},{\bf X})>0]\leq
\alpha\ .$$
This means that the type I error probability is controlled uniformly over all
variance $\sigma^2$. Similarly, we want to control the type II error
probabilities uniformly over all variances. The $\delta$-separation distance
$\rho_{F,U}[\Phi_{\alpha},k,{\bf X}]$ of
$\Phi_{\alpha}$ over $\Theta[k,p]$ for unknown variance is defined by
\begin{eqnarray}\label{distance_variance_inconnue_fixe}
\rho_{F,U}[\Phi_{\alpha},k,{\bf X}]:=
\inf\Bigg\{\rho>0,\, \inf_{\begin{array}{c}
\scriptstyle \sigma>0,\ \theta_0\in\Theta[k,p],\\                           
\scriptstyle \|{\bf X}\theta_0\|_n\geq
\sqrt{n}\rho\sigma                          
 \end{array}
}\mathbb{P}_{\theta_0,\sigma}[\Phi_{\alpha}=1]\geq
1-\delta\Bigg\}\ .
\end{eqnarray}
Hence, $\rho_{F,U}[\Phi_{\alpha},k,{\bf X}]$ corresponds to
the minimal distance such that the hypotheses $\{\theta_0=0_p\text{ and }\sigma>0\}$ and $\{
\theta_0\in\Theta[k,p]\text{ and }\sigma>0\ ,\ \|{\bf
X}\theta_0\|_n^2\geq n\rho_{F,U}^2[\Phi_{\alpha},k,{\bf X}]\sigma^2\}$
are
well separated by the test $\Phi_{\alpha}$. Taking the infimum over all level
$\alpha$ tests, we get the $(\alpha,\delta)$ minimax
separation distance over $\Theta[k,p]$ with design ${\bf X}$ and unknown
variance is
\begin{eqnarray}\label{distance_minimax_variance_inconnue_fixe}
\rho^{*}_{F,U}[k,{\bf X}]:=
\inf_{\Phi_{\alpha}}\rho_{F,U}[\Phi_{\alpha},k,{\bf X}]\ .
\end{eqnarray}
Finally, $\rho^{*}_{F,U}[k]:=\sup_{\bf
X}\rho^{*}_{F,U}[k,{\bf X}]$ corresponds to the $(\alpha,\delta)$-minimax
separation
distance over $\Theta[k,p]$ with the ``worst-case designs".

In the Gaussian design, we define $\rho_{R,U}[\Phi_{\alpha},k,\Sigma]$,
$\rho^*_{R,U}[k,\Sigma]$, and $\rho^*_{R,U}[k]$ analogously to
(\ref{distance_variance_inconnue_fixe}) and
(\ref{distance_minimax_variance_inconnue_fixe}) by replacing the norm $\|{\bf
X}\theta_0\|_n/\sqrt{n}$ by $\|\sqrt{\Sigma}\theta_0\|_p$.

\subsubsection{Gaussian design}

Minimax bounds have been proved in ~\cite{villers1} in the non ultra-high dimensional setting. The next theorem encompasses high dimensional and ultra-high dimensional settings.

\begin{thrm}\label{thrm_minimax_testing}
Suppose that $\alpha+\delta\leq 53\%$ and that $p \geq n\geq
8\log(2/\delta)$.
For any $1\leq k\leq \lfloor p^{1/3}\rfloor$, the
$(\alpha,\delta)$-minimax separation distance over $\Theta[k,p]$ with covariance
$I_p$  and unknown variance satisfies
\begin{eqnarray}\label{borne_minoration_test}
(\rho^*_{R,U}[k,I_p])^2\geq  C_1\frac{k}{n}\log\left(p\right)
\exp\left[C_2\frac{k}{n}\log\left(p\right)\right]\ . 
\end{eqnarray}
For any $1\leq k\leq n/2$ and  any covariance $\Sigma$, we have
\begin{eqnarray}\label{borne_majoration_test}
 (\rho^*_{R,U})^2[k,\Sigma] \leq 
C_1(\alpha,\delta)\frac{k}{n}\log\left(\frac{ep}{k}\right)
\exp\left[C_2(\alpha,\delta)\frac{k}{n}\log\left(\frac{ep}{k}\right)\right] \ .
\end{eqnarray}
Furthermore, this upper bound is simultaneously achieved for all $k$ and $\Sigma$ by a procedure $T_{\alpha}$ (defined in Section \ref{section_test}).
\end{thrm}

\begin{remark} {\bf [Minimax adaptation]}
It follows from Theorem \ref{thrm_minimax_testing} that, under unknown variance, adaptation to the sparsity is possible  and that   the minimax separation distance $(\rho^*_{RU}[k])^2$ over $\Theta[k,p]$ is of order 
\begin{equation}\label{vitesse_minimax_adatatation_eq}
C_1(\alpha,\delta)\frac{k}{n}\log\left(p\right)
\exp\left[C_2(\alpha,\delta)\frac{k}{n}\log\left(p\right)\right]\ .
\end{equation}
\end{remark}

\begin{remark}
The condition $k\leq p^{1/3}$ can be replaced by $k\leq p^{1/2-\gamma}$ with
$\gamma>0$, the only difference being that the constants involved in (\ref{borne_minoration_test}) would depend on $\gamma$. 
These conditions are not really restrictive for a sparse
high-dimensional regression since the usual setting is $k\leq n\ll p$. 

Note $k\leq p^{3}$ implies that $\log(p)\leq 3/2\log(p/k)\leq 3\log(p/k^2)$ so that 
we cannot distinguish terms $C_1\log(p)$ from  $C_2\log(p/k^2)$ or $C_3\log(p/k)$. As a consequence (\ref{vitesse_minimax_adatatation_eq}) does not necessarily capture the right dependency on $k$ in the logarithmic terms. This observation also holds for all the next results that require $k\leq p^{1/3}$.
\end{remark}

\begin{remark} {\bf [Dependent design]} As for the known variance case, we have $\rho^*_{R,U}[k,I_p]\geq C(\alpha,\delta)\rho^*_{R,U}[k]$, that is the testing problem is more complex for an independent design than for a correlated design. 
For some covariance matrices $\Sigma$, the minimax separation distance with
covariance $\Sigma$ is much smaller than $\rho^*_{R,U}[k,I_p]$. Verzelen and Villers~\cite{villers1} provide such an example of a matrix $\Sigma$ in (see Propositions 8 and 9). However, the arguments used in the proof of their example are not generalizable  to 
other covariances. In fact, 
the computation of sharp minimax bounds that 
capture the dependency of $\rho^*_{R,U}[k,\Sigma]$ on $\Sigma$
remains an open problem. 
\end{remark}

\subsubsection{Fixed design}

Ingster et al.~\cite{Ingster10} derive the asymptotic minimax separation distance for some specific design when $k\log(p)/n$ goes to $0$. Here, we provide the non asymptotic counterpart that encompass all the regimes.

\begin{prte}\label{prte_minoration_minimax_test_fixed_unknown_variance}
Assume that $\alpha+\delta\leq 26\%$ and that  $p\geq n\geq C(\alpha,\delta)$.
For any $1\leq k\leq \lfloor p^{1/3}\rfloor$, there exist some $n\times
p$ designs ${\bf X}$ such that 
\begin{eqnarray}\label{borne_minoration_test_fixed_unknown}
(\rho^*_{F,U}[k,{\bf X}])^2\geq 
C_1\frac{k}{n}\log\left(p\right)
\exp\left[C_2\frac{k}{n}\log\left(p\right)\right]\ .
\end{eqnarray}
For any $1\leq k\leq n/2$ and any $n\times p$ design ${\bf X}$, we have
\begin{eqnarray}\label{borne_majoration_test_fixed_variance_inconnue}
 (\rho_{F,U}^*[k,{\bf X}])^2 \leq 
C_1(\alpha,\delta)\frac{k}{n}\log\left(\frac{ep}{k}\right)
\exp\left[C_2(\alpha,\delta)\frac{k}{n}\log\left(\frac{ep}{k}\right)\right] \ .
\end{eqnarray}
Furthermore, this upper bound is simultaneously achieved for all $k$ and ${\bf X}$ by a procedure $T_{\alpha}$ (defined in Section \ref{section_test}).
\end{prte}

Again, we observe a phenomenon analogous to the random design case.

\subsection{Comparison between known and unknown variance}

There are three regimes depending on $(k,p,n)$. They are depicted on Figure \ref{figure01}:
\begin{enumerate}
 \item ${\bf k\boldsymbol{\log}(p)\leq \sqrt{n}}$. The minimax separation
distances are of the same order for known and unknown $\sigma^2$. The minimax
distance $k\log(p)/n$ is also of the same order as the minimax risk of
prediction.
\item  ${\bf \boldsymbol{\sqrt{n}}\leq k\boldsymbol{\log}(p)\leq n}$. If
$\sigma^2$ is known, the minimax separation distance is always of order
$1/\sqrt{n}$. In such a case, an optimal procedure amounts to test the
hypothesis $\{\mathbb{E}[\|{\bf Y}\|_n^2]=n\sigma^2\}$ against $\{\mathbb{E}[\|{\bf Y}\|_n^2]>  n\sigma^2\}$ using the
 statistic $\|{\bf Y}\|_n^2/\sigma^2$. If $\sigma^2$ is unknown, the statistic $\|{\bf Y}\|_n^2/\sigma^2$ is not available and the minimax separation distance behaves like 
$k\log(p)/n$.
\item ${\bf k\boldsymbol{\log}(p)\geq n}$. If $\sigma^2$ is unknown, the
minimax separation distance blows up. It is of order $(p/k)^{Ck/n}$. 
Consequently, the problem of testing $\{\theta_0=0_p\}$ becomes extremely difficult
in this setting.
\end{enumerate}

\section{Prediction}\label{section_prediction}

In contrast to the testing problem, the minimax risks of prediction (${\bf
P_2}$) exhibit really different behaviors in fixed and in random design. The big picture is summarized in Figure \ref{figure02}. We recall that the minimax risks $\mathcal{R}_F[k,{\bf X}]$, $\mathcal{R}_F[k]$, $\mathcal{R}_R[k,\Sigma]$, and $\mathcal{R}_R[k]$ are defined in Section \ref{section_main_prediction}.

\subsection{Gaussian design}

\begin{prte}{\bf [Minimax lower bound for
prediction]}\label{prte_prediction_minoration_minimax}
Assume that $p\geq C$. For any $1\leq k\leq  \lfloor p^{1/3}\rfloor$, we
have
\begin{eqnarray}\label{minoration_minimax_prediction}
\mathcal{R}_{R}[k,I_p]\geq  
C\frac{k}{n}\log\left(\frac{ep}{k}\right)
\exp\left\{C_2\frac{k}{n}\log\left(\frac{ep}{k}\right)\right\}\ . 
\end{eqnarray}
\end{prte}

\begin{remark}{\bf [General covariances
$\Sigma$]}\label{remarque_covariance_generale}
The lower bound (\ref{minoration_minimax_prediction}) is only stated
for the identity covariance $\Sigma=I_p$. For general covariance matrices
$\Sigma$, we have
\begin{eqnarray}\label{eq_minoration_random_design_restricted}
\mathcal{R}_{R}[k,\Sigma]\geq
C\frac{\varPhi_{2k,-}(\sqrt{\Sigma})}{\varPhi_{2k,+}(\sqrt{\Sigma})}
\times \frac{k}{n}\log\left(\frac{ep}{k}\right)\ ,
\end{eqnarray}
for any $k\leq n\leq p/2$. 
This statement has been proved in~\cite{verzelen_regression}
(Proposition 4.5) in the special case of restricted isometry, but the proof
straightforwardly extends to restricted eigenvalue conditions. For $\Sigma=I_p$, the
lower bound (\ref{eq_minoration_random_design_restricted}) does not capture the
elbow effect in an ultra-high dimensional setting (compare with
(\ref{minoration_minimax_prediction})). 
\end{remark}

\begin{thrm}{\bf [Minimax upper bound]}
\label{thrm_prediction_majoration_minimax}
Assume that $n\geq C$. There exists an estimator $\widetilde{\theta}^V$ (defined in Section \ref{section_theta_V}) such that the following holds:
\begin{enumerate}
 \item The computation of $\widetilde{\theta}^V$ does not require the knowledge of $\sigma^2$ or $k$.
\item For any covariance
$\Sigma$, any $\sigma>0$, any $1\leq k\leq \lfloor (n-1)/4\rfloor$, and any
$\theta_0\in\Theta[k,p]$ we
have
\begin{eqnarray}
 \mathbb{E}_{\theta_0,\sigma}\left[\|\sqrt{\Sigma}(\widetilde{\theta}
^V-\theta_0)\|_p^2\right ]
\leq C_1 \frac{k}{n}\log\left(\frac{ep}{k}\right)
\exp\left\{C_2\frac{k}{n}\log\left(\frac{ep}{k}\right)\right\}\sigma^2\ . 
\label{inegalite_oracle}
\end{eqnarray}	
\end{enumerate}

\end{thrm}
In contrast to similar results such as Theorem 1 in Giraud~\cite{giraud08}
or Theorem 3.4 in Verzelen~\cite{verzelen_regression}, we do not restrict
 $k$ to be smaller than $n/(2\log p)$, that is we encompass both high-dimensional and ultra-high dimensional setting. The proof of the
theorem is based on a new deviation inequality for the spectrum of
Wishart matrices stated in Lemma \ref{lemma_concentration_vp_wishart}.

\begin{remark} 
 {\bf [Minimax risk]}
We derive from Theorem
\ref{thrm_prediction_majoration_minimax}  and Proposition
\ref{prte_prediction_minoration_minimax} that the minimax
risk $\mathcal{R}_{R}[k]$ is of order
$$  C_1\frac{k}{n}\log\left(\frac{ep}{k}\right)
\exp\left\{C_2\frac{k}{n}\log\left(\frac{ep}{k}\right)\right\}
\ .$$
If $k\log(p/k)$ is small compared to $n$, the minimax risk of estimation is
of order $Ck\log(p/k)/n$.  In an ultra-high dimensional setting, we again
observe a blow up.
\end{remark}

\begin{remark} {\bf [Adaptation to sparsity and the variance]} The estimator $\widetilde{\theta}^V$ does not
requires the knowledge of $k$ and 
of the variance $\sigma^2=\var(Y|X)$. It follows that $\widetilde{\theta}^V$ is minimax
adaptive to all $1\leq k\leq p^{1/3}\wedge [(n-1)/4]$ and to all $\sigma^2>0$. As a consequence,
adaptation to the sparsity and to the variance is possible for this problem.
\end{remark}

\begin{remark} {\bf [Dependent design]} The risk upper bound of
$\widetilde{\theta}^V$
stated in Theorem \ref{thrm_prediction_majoration_minimax}  is valid for any
covariance matrix $\Sigma$ of the covariance $X$. In contrast, the minimax lower
bound of Theorem \ref{thrm_minimax_testing} is restricted to the identity
covariance. This implies that the minimax prediction risk for a general
matrix $\Sigma$ is at worst of the same order as in the independent case: there
exists a universal constant $C>0$ such that for all covariance $\Sigma$,
$$\mathcal{R}_{R}[k,I_p]\geq C\mathcal{R}_{R}[k]\ .$$

 In Remark \ref{remarque_covariance_generale}, we have stated a
minimax lower bound for prediction that depends on the restricted eigenvalues
of $\Sigma$. Fix some $0<\gamma<1$. If we consider some covariance matrices
$\Sigma$ such that
$\varPhi_{2k,-}(\sqrt{\Sigma})/\varPhi_{2k,+}(\sqrt{\Sigma})\geq 1-\gamma$ , the
minimax lower bound (\ref{eq_minoration_random_design_restricted}) and the upper
bound (\ref{inegalite_oracle}) match up to a constant $C(\gamma)$.
In general, the lower bound (\ref{eq_minoration_random_design_restricted}) and the upper bound (\ref{inegalite_oracle}) do not exhibit the
same dependency with respect to $\Sigma$, especially when
$\varPhi_{2k,-}(\sqrt{\Sigma})/\varPhi_{2k,+}(\sqrt{\Sigma})$ is close to zero.
\end{remark}

\subsection{Fixed design}

\subsubsection{Known variance}

The minimax prediction risk with known variance has been studied in Raskutti et al.~\cite{raskwain09} and Rigollet and Tsybakov~\cite{tsyrig10} (see also~\cite{abramovich10,zhangminimax}). For any design ${\bf X}$ and any $1\leq k\leq n$, these authors have proved that the minimax risk
$\mathcal{R}_F[k,{\bf X}]$ satisfies
\begin{eqnarray}\label{eq_minoration_fixed_design_restricted}
C_1
\inf_{s\leq k} \frac{\varPhi_{2s,-}({\bf X})}{\varPhi_{2s,+}({\bf X})}
\frac{s}{n}\log\left(\frac{ep}{s}\right)\leq  \mathcal{R}_F[k,{\bf X}]  \leq C_2
\frac{k}{n}\log\left(\frac{ep}{k}\right)
\ .
\end{eqnarray}
Next, we bound the  supremum $\sup_{{\bf X}}R_F[k,{\bf X}]$ and we study the possibility of adaptation to the sparsity.

\begin{prte}\label{prte_minimax_prediction_fixed_design} 
For any $1\leq k\leq n$, the supremum $\sup_{{\bf X}}\mathcal{R}_F[k,{\bf X}]$ is lower bounded as follows
\begin{equation}\label{eq_minoration_fixed_design}
 \mathcal{R}_F[k]\geq C
\left[\frac{k}{n}\log\left(\frac{ep}{k}\right)\wedge 1
\right]\ .
\end{equation}
Assume that $p\geq n$. There exists an estimator $\tilde{\theta}^{BM}$ (defined in Section \ref{section_theta_BM}) which satisfies
\begin{equation}\label{eq_adaptation}
\sup_{{\bf X}}\sup_{\theta_0\in\Theta[k,p]}\mathbb{E}_{\theta_0,\sigma}\left[\|{\bf X}(\widehat{\theta}^{BM}-\theta_0)\|_n^2\right]/(n\sigma^2)\leq C
\left[\frac{k}{n}\log\left(\frac{ep}{k}\right)\wedge 1
\right]\ ,
\end{equation}
for any $1\leq k\leq n$.
\end{prte}

This upper bound (\ref{eq_adaptation})  is a consequence of Birg\'e and Massart~\cite{BM01}.

\begin{remark}
If $k\log(p/k)$ is small compared to $n$, the minimax risk  is of order
$Ck\log(p/k)/n$.  In an ultra-high dimensional setting, this minimax risk
remains close to one. This corresponds (up to renormalization) to the minimax
risk of estimation of the vector $\mathbb{E}[{\bf Y}]$ of size $n$ . As a
consequence, the sparsity assumption does not play anymore a role in a
ultra-high dimensional setting. From (\ref{eq_adaptation}), we derive that adaptation to the sparsity is possible  when the variance $\sigma^2$ is known.
\end{remark}

\begin{remark}
 {\bf [Dependency of $\mathcal{R}_F[k,{\bf X}]$ on ${\bf X}$]} For designs ${\bf X}$, such that 
the ratio $\varPhi_{2k,-}({\bf X})/\varPhi_{2k,+}({\bf X})$ is close to one,
the lower bounds and upper bounds of (\ref{eq_minoration_fixed_design_restricted})  agree with each other. 
This is for instance the case of the realizations (with high probability) of a Gaussian standard independent design (see the proof of Proposition \ref{prte_minimax_prediction_fixed_design} for more details).

However,  the dependency of the minimax lower bound in (\ref{eq_minoration_fixed_design_restricted})
on ${\bf X}$ is not sharp when the ratio $\varPhi_{2k,-}({\bf
X})/\varPhi_{2k,+}({\bf X})$ is away from one. Take for instance an
orthogonal design with
$p=n$
and duplicate the last column. Then, the lower bound
(\ref{eq_minoration_fixed_design_restricted}) for this new design ${\bf X}$ is
$0$ while the minimax risk is of order $k\log(p/k)/n$. 

Similarly, the dependency of the minimax upper bound in (\ref{eq_minoration_fixed_design_restricted}) on ${\bf X}$ is not sharp. For very specific design, it is possible to obtain a minimax risk $\mathcal{R}_F[k,{\bf X}]$ that is much  smaller than $k/n\log(p/k)\wedge 1$ (see Abramovich and Grinshtein~\cite{abramovich10}).
\end{remark}

\begin{remark}{\bf [Comparison with $l_1$ procedures]}  The designs ${\bf X}$
for which $l_1$  procedures such as the Lasso or the Dantzig selector are proved
to perform well require that $\varPhi_{2k,-}({\bf
X})/\varPhi_{2k,+}({\bf X})$ is close to one. It is interesting to notice that
these designs ${\bf X}$ precisely correspond to situations where the minimax
risk is close to its maximum $k\log(p/k)/n$ (see Equation
(\ref{eq_minoration_fixed_design_restricted})). We refer to \cite{raskwain09} for a more complete discussion.\end{remark}

\begin{remark}\label{remarque_matrices}
We easily  retrieve from (\ref{eq_minoration_fixed_design_restricted}) a 
result of asymptotic geometry first observed by Baraniuk et
al.~\cite{baraniuk08} in the special of restricted isometry
property~\cite{candes05}. For any $0<\delta\leq 1$, there exists a constant
$C(\delta)>0$ such that no
$n\times p$ matrix ${\bf X}$ can fulfill  $\varPhi_{k,-}({\bf
X})/\varPhi_{k,+}({\bf X})\geq \delta$ if $k(1+\log(p/k))\geq
C(\delta)n$.
\end{remark}

\begin{proof}
If $\varPhi_{2k,-}({\bf
X})/\varPhi_{2k,+}({\bf X})\geq \delta$, then 
$ \mathcal{R}_{F}[k,{\bf X}] \geq
C\delta k\log\left(ep/k\right)/n$.~\\
We also have
$ \mathcal{R}_{F}[k,{\bf X}]\leq  \mathcal{R}_{F}[p,{\bf X}]\leq 1$. The last inequality follows from the risk of an estimator $\widehat{\theta}_n\in\arg\min_{\theta\in\mathbb{R}^p}\|{\bf Y}-{\bf X}\theta\|_n^2$.
Gathering these two bounds allows to conclude.
\end{proof}

\subsubsection{Unknown variance}
We now consider the problem of prediction when the variance $\sigma^2$ is
unknown.

\begin{prte}\label{prte_majoration_prediction_fixed_unknown}
 For any $1\leq k\leq n$, there exists an estimator $\widehat{\theta}^{(k)}$ that does not require the knowledge of $\sigma^2$ such that
\begin{equation}\label{eq_majoration_prediction_fixed_unknown}
 \sup_{{\bf X}}\sup_{\sigma>0}\sup_{\theta_0\in\Theta[k,p]}\mathbb{E}_{\theta_0,\sigma}\left[\|{\bf X}(\widehat{\theta}^{(k)}-\theta_0)\|_n^2\right]/(n\sigma^2)\leq C
\left[\frac{k}{n}\log\left(\frac{ep}{k}\right)\wedge 1
\right]\ .
\end{equation}
\end{prte}
Thus, the optimal risk of prediction over $\Theta[k,p]$ remains of the same order
for known and unknown $\sigma^2$.

Let us now study to what extent
adaptation to the sparsity is possible when the variance
$\sigma^2$ is unknown.  In order to get some ideas let us provide risk bounds for two procedures that do not require the knowledge of $\sigma$: the estimator $\widetilde{\theta}^V$ already studied for Gaussian design (defined in Section \ref{section_theta_V})
 and the estimator $\widehat{\theta}_n$ defined by  $\widehat{\theta}_n\in\arg\min_{\theta\in\mathbb{R}^p}\|{\bf Y}-{\bf X}\theta\|_n^2$.

\begin{prte}{\bf [Risk bound for $\widetilde{\theta}^{V}$ and $\widehat{\theta}_n$]}
 \label{prte_risque_baraudgiraud}
Assume that $n\geq 14$.
For any $1\leq k\leq \lfloor(n-1)/4\rfloor$, the maximal risk of $\widehat{\theta}^V$ over $\Theta[k,p]$ is upper bounded as follows 
\begin{eqnarray}\label{inegalite_oracle_baraud}
\sup_{{\bf X}}\sup_{\sigma>0} \sup_{\theta_0\in\Theta[k,p]}\mathbb{E}_{\theta_0,\sigma}\left[\|{\bf
X}(\tilde{\theta}^{V}-\theta_0)\|_n^2\right]/(n\sigma^2)\leq C_1
\frac{k}{n}\log\left(\frac{ep}{k}\right)\exp\left[C_2\frac{k}{n}\log\left(\frac{ep
} { k }
\right)\right]\sigma^2\ . 
\end{eqnarray}
For any $1\leq k\leq n$, the maximal risk of $\widehat{\theta}_n$ over $\Theta[k,p]$ is upper bounded as follows 
\begin{eqnarray}\label{inegalite_oracle_plein}
 \sup_{{\bf X}}\sup_{\sigma>0}\sup_{\theta_0\in\Theta[k,p]}\mathbb{E}_{\theta_0,\sigma}\left[\|{\bf
X}(\widehat{\theta}_n-\theta_0)\|_n^2\right]/(n\sigma^2)\leq 1\ . 
\end{eqnarray}
\end{prte}
The risk bound (\ref{inegalite_oracle_baraud}) is also satisfied by the procedure of Baraud et al.~\cite{BGH09}. The proof of (\ref{inegalite_oracle_baraud}) is a consequence  of one of their results.

\begin{remark}
As a consequence, $\tilde{\theta}^{V}$ simultaneously achieves the minimax
risk over all $\Theta[k,p]$ for all ${k\leq \lfloor(n-1)/4\rfloor}$ such that
$k(1+\log(p)/k)\leq n$.
In an ultra-high dimensional setting, the maximum risk of $\tilde{\theta}^{V}$
over
$\Theta[k,p]$ is controlled by $(ep/k)^{Ck/n}$ while the minimax risk is
smaller than $1$. If the upper bound (\ref{inegalite_oracle_baraud}) is sharp then this would imply that $\tilde{\theta}^V$ is not adaptive to the sparsity  in an ultra-high dimensional setting. 

In contrast, $\widehat{\theta}_n$ is minimax adaptive over all
$\Theta[k,p]$ such that $k(1+\log(p)/k)\geq n$, but its behavior is suboptimal in a non-ultra-high dimensional setting.
\end{remark}

In order to get an estimator that is adaptive to all indexes $k$, we would need to merge the properties of $\widetilde{\theta}^V$ (for non-ultra-high dimensional cases) and of $\widehat{\theta}_n$ (for ultra-high dimensional cases). The following proposition tells us that it is in fact impossible.

\begin{prte}{\bf [Adaptation to the sparsity is impossible under unknown variance]}\label{prte_adaptation_impossible}
Consider any $p\geq n\geq C_1$ and $1\leq k\leq \lfloor p^{1/3}\rfloor $ such that
$k\log(ep/k)\geq C_2 n$. There exists a design
${\bf X}$ of size $n\times p$ such that
for any estimator $\widehat{\theta}$, we have either
$$\sup_{\
\sigma>0}\mathbb{E}_{0_p,\sigma}\left[\|{ \bf
X}(\widehat{\theta}-0_{p})\|_n^2/(n\sigma^2)\right] > C\ ,
$$
$$\text{or}\hspace{2cm} \sup_{\sigma>0}\ \sup_{\theta_0\in\Theta[k,p]}\mathbb{E}_{\theta_0,\sigma}\left[\|{\bf
X}(\widehat{\theta}-\theta_0)\|_n^2/(n\sigma^2)\right]
 > \exp\left[C\frac{k}{n}\log\left(\frac{ep}{k}\right)\right]\
. \hspace{2cm}$$
As a benchmark, we recall the minimax upper bounds:
$$\mathcal{R}_{F}[1]\leq C_1\frac{\log(p)}{n}\quad \text{ and }\quad
\mathcal{R}_{F}[k]\leq C_2\left[\frac{k}{n}\log\left(\frac{ep}{k}\right) \wedge 1\right]\ .$$
\end{prte}

The proof of proposition \ref{prte_adaptation_impossible} is based on the minimax lower
bounds (\ref{borne_minoration_test_fixed_unknown}) for the testing problem
(${\bf P_1}$) under unknown variance. The proof uses designs ${\bf X}$ that are realizations of standard Gaussian designs.

\begin{remark}
 In the setup of Proposition \ref{prte_adaptation_impossible}, any estimator
$\widehat{\theta}$ that does not require the knowledge of $k$ and $\sigma^2$ has to pay at least one of
these two
prices:
\begin{enumerate}
\item The estimator $\widehat{\theta}$ does not use the sparsity of the true
parameter $\theta_0$. Its risk for estimating $0_p$ is of the same order as the
minimax risk over $\mathbb{R}^p$. The estimator $\widehat{\theta}_n$ has this
drawback.
\item For any $1\leq k\leq p^{1/3}$, we have
\begin{eqnarray*}
 \sup_{\bf X}\sup_{\sigma>0}\sup_{\theta_0\in\Theta[k,p]}
\mathbb{E}_{\theta_0,\sigma}\left[\|{\bf
X}(\widehat{\theta}-\theta_0)\|_n^2/(n\sigma^2)\right]
\geq
C_1\frac{k}{n}\log\left(\frac{ep}{k}\right)\exp\left[C_2\frac{k}{n}
\log\left(\frac{ep}{k}\right)\right ]\ .
\end{eqnarray*}
This is the price for adaptation when $\sigma^2$ is unknown.
The estimator $\widetilde{\theta}^{V}$ exhibits this behavior.
\end{enumerate}
As a conclusion, it is impossible to merge the qualities of
$\widetilde{\theta}^{V}$ and of $\widehat{\theta}_n$. 

The best prediction risk that can be achieved by a procedure that aim to adaptation to the sparsity is of order
$$\frac{k}{n}\log\left(\frac{p}{k}\right)\exp\left[C\frac{k}{n}\log\left(p/k\right)\right]\ .$$
In other words, the unavoidable loss for adaptation for unknown variance is a factor  $\exp[Ck/n\log(p/k)]$ 
In this sense,  the estimator $\widetilde{\theta}^V$ (and as a byproduct the procedure of Baraud et al.~\cite{BGH09}) achieves the optimal prediction risk  under
unknown variance and unknown sparsity.
\end{remark}

\vspace{0.5cm}

In conclusion, the minimax risks of prediction are of the same order for
fixed and Gaussian design and for known and unknown variance when
$k\log(p/k)$ is small compared to $n$. In an ultra-high dimensional setting, the
minimax risks behave differently. For Gaussian design, the minimax risk is of the
order $(p/k)^{Ck/n}$. In contrast, the minimax risk of prediction remains
smaller than one for fixed design regression with known variance. When the sparsity and the variance are unknown, there is a price to pay for adaptation under fixed design. All these behaviors are depicted on Figure \ref{figure02}.

\section{Inverse problem and support estimation}\label{section_inverse}

\subsection{Minimax risk of estimation}

We recall that the  minimax risks of estimation for the inverse problem
$\mathcal{RI}_F[k,{\bf X}]$, $\mathcal{RI}_F[k]$, $\mathcal{RI}_R[k,\Sigma]$, and $\mathcal{RI}_R[k]$ have been defined in Section \ref{section_main_result_inverse}.

\subsubsection{Fixed design}

First, we consider the problem (${\bf P_3}$) for a fixed design regression
model. The minimax risk of estimation over $\Theta[k,p]$
with a design ${\bf X}$ is noted  $\mathcal{RI}_F[k,{\bf X}]$ and is defined in
(\ref{defi_risque_minimax_inverse}). Raskutti et al.~\cite{raskwain09} have recently provided the following bounds
\begin{eqnarray}\label{minoration_inverse_design_fixe_rask}
 C_1\left[
\frac{k\log(ep/k)}{\varPhi_{2k\wedge p,+}\left({\bf X}\right)}\right] \leq \mathcal{RI}_F[k,{\bf X}]\leq C_2
\frac{k\log\left(ep/k\right)}{\varPhi_{2k\wedge p,-}({\bf X})}\ ,
\end{eqnarray}
that holds for any fixed design ${\bf X}$ and any $1\leq k\leq n$. The lower and upper bounds match up to the factor $\varPhi_{2k\wedge p,+}({\bf X})/\varPhi_{2k\wedge p,-}({\bf X})$. The upper bound is achieved by least-squares estimator over $\Theta[k,p]$~\cite{raskwain09}.
 If the restricted eigenvalues of ${\bf X}$  are close to one, then
the minimax risk is of order $k\log(ep/k)$. 
Next, we improve the lower bound in (\ref{minoration_inverse_design_fixe_rask}) in order to grasp  the behavior of the minimax risk for non orthogonal design.
\begin{prte}\label{prte_minoration_inverse_fixe}
For any design ${\bf X}$ and any $1\leq k\leq n$, we have 
\begin{eqnarray}\label{minoration_inverse_design_fixe}
\mathcal{RI}_F[k,{\bf X}]\geq C\left[
\frac{1}{\varPhi_{2k\wedge p,-}\left({\bf X}\right)}\vee
\frac{k\log(ep/k)}{\varPhi_{1,+}\left({\bf X}\right)}\right]\ .
\end{eqnarray}
\end{prte}
In order to interpret these bounds let us restrict ourselves to design ${\bf X}$ such that each column has $\sqrt{n}$ norm, as justified in
Section \ref{section_main_result_inverse}. The collection of such designs is noted $\mathcal{D}_{n,p}$. Observe that ${\bf X}\in\mathcal{D}_{n,p}$ enforces $\varPhi_{1,+}\left({\bf X}\right)=n$.

In the sequel, we are interested in the smallest minimax risk $\mathcal{RI}_F[k,{\bf X}]$ that is achievable if we can choose the $n\times p$ design ${\bf X}\in\mathcal{D}_{n,p}$, that is we want to bound 
$\mathcal{RI}_F[k]= \inf_{{\bf
X}\in\mathcal{D}_{n,p}} \mathcal{RI}_F[k,{\bf X}]$. The  minimax risk  $\mathcal{RI}_F[k]$ tells us the intrinsic difficulty of estimating a $k$ sparse vector of size $p$ with $n$ observations.

\begin{prte} \label{cor_minoration_inverse_dimension_pas_trop_grande}

\begin{enumerate}
~
 \item Assume that $k[1+\log(p/k)]\leq Cn$. Then, we have
\begin{equation}\label{equation_minoration_inversion_raisonnable}
C_1
\frac{k}{n}\log\left(\frac{ep}{k}\right)\leq
\mathcal{RI}_F[k]\leq C_2
\frac{k}{n}\log\left(\frac{ep}{k}\right)\ .
\end{equation}
This bound is for instance achieved for designs ${\bf X}$ that
are realizations (with a high probability)  of normalized standard Gaussian design.
\item For any design ${\bf X}\in\mathcal{D}_{n,p}$ and any
$k\leq n\wedge p/2$, we have 
\begin{eqnarray}\label{minoration_valeur_propre_restreinte_design_fixe}
 \varPhi_{2k,-}({\bf X})\leq C_1n\left(\frac{k}{ep}\right)^{C_2k/n}\ .
\end{eqnarray}
\item For any $k\leq n/4\wedge p/2$, we have  
\begin{equation}\label{equation_equivalent_asymptotique_inverse_minimax}
C_1\left[\frac{k}{n}\log\left(\frac{ep}{k}\right)\vee \frac{1}{n}\exp\left\{C_4\frac{k}{n}\log\left(\frac{p}{k}\right)\right\}\right]\leq \mathcal{RI}_F[k]\leq C_2\frac{k}{n}\log\left(\frac{p}{k}\right)\exp\left[C_3\frac{k}{n}\log\left(\frac{p}{k}\right)\right]\ .
\end{equation}
\end{enumerate}
\end{prte}

\begin{remark} The bound (\ref{equation_minoration_inversion_raisonnable}) tells us that the best minimax risk that is achievable in a non-ultra-high dimensional setting is of order $k\log(ep/k)/n$. The Lasso achieves the (almost optimal) risk bound $k\log(p)/n$ under  some assumptions on the design matrix.
\end{remark}

\begin{remark}
The lower bound (\ref{minoration_valeur_propre_restreinte_design_fixe}) is of geometric nature. Combined with (\ref{minoration_inverse_design_fixe}), it implies the lower bound of (\ref{equation_equivalent_asymptotique_inverse_minimax}). In an ultra-high dimensional setting, it is not  possible to build a design ${\bf X}$ such that $\varPhi_{2k,+}\left({\bf X}\right)/\varPhi_{2k,-}\left({\bf
X}\right)$ is  close to one (see Remark \ref{remarque_matrices}). In fact, the quantity $\varPhi^{-1}_{2k,-}({\bf X})$ blows up because of geometric constrains. When $k[1+\log(p/k)]$ is larger compared to $n\log(n)$, both bounds in (\ref{equation_equivalent_asymptotique_inverse_minimax}) are comparable and the minimax risk is of order $\exp[Ck/n\log(p/k)]$.  As a consequence, the inverse problem becomes extremely difficult in an ultra-high dimensional setting.
\end{remark}

\begin{remark}
 While the quantity $k\log(p/k)$ in (\ref{equation_minoration_inversion_raisonnable}) is due to the ``size" of the parameter space
$\Theta[k,p]$, the exponential term of the minimax risk in ultra-high dimension is essentially driven by
geometrical constrains on the design ${\bf X}$.
\end{remark}

\begin{prte}[Adaptation to the sparsity and the variance]\label{prte_majoration_probleme_inverse_adaptation} As in the prediction case, we consider the estimator $\widetilde{\theta}^V$ (defined in Section \ref{section_theta_V}). Assume that $p\geq 2n$. For any design ${\bf X}$, any $\sigma>0$, any $1\leq k\leq \lfloor (n-1)/4\rfloor$, and any $\theta_0\in\Theta[k,p]$, we have  
\begin{equation}\label{eq_prte_majoration_probleme_inverse_adaptation}
 \frac{\|\widetilde{\theta}^V-\theta_0\|_p^2}{\sigma^2}\leq C_1 \frac{k}{\varPhi_{3k,-}({\bf X})}\log\left(\frac{ep}{k}\right)\exp\left[C_2\frac{k}{n}\log\left(\frac{ep}{k}\right)\right]\ ,
\end{equation}
with probability larger than $1-e^{-n}-C/p$.
\end{prte}

\begin{remark}
 Although the bound (\ref{eq_prte_majoration_probleme_inverse_adaptation}) is in probability and not in expectation, it  suggests that adaptation to the sparsity and to the variance are possible.
\end{remark}

\subsubsection{Random design}

Let us turn to the Gaussian design case. We are interested in bounding  $\mathcal{RI}_R[k,\Sigma]$ and $\mathcal{RI}_R[k]$ as defined in (\ref{defi_risque_minimax_inverse_random}).

\begin{prte}\label{prte_minimax_inverse_random}
For any $1\leq k\leq (n-1)/4$, and any covariance $\Sigma$ we have
\begin{equation}\label{minoration_inverse_design_random_rask}
 C_1\left[
\frac{1}{n\varPhi_{2k\wedge p,-}(\sqrt{\Sigma})}\vee
\frac{k\log(ep/k)}{n\varPhi_{1,+}(\sqrt{\Sigma})}\right]\leq \mathcal{RI}_R[k,\Sigma]\leq 
C_2
\frac{k\log\left(ep/k\right)}{n\varPhi_{2k\wedge p,-}(\sqrt{\Sigma})}\exp\left[C_3\frac{k}{n}\log\left(\frac{ep}{k}\right)\right]\ .
\end{equation}
As long as $k[1\log(p/k)]\leq n$, we derive that  $\mathcal{RI}_R[k]:= \inf_{\Sigma \in \mathcal{S}_{p}}\mathcal{RI}_R[k,\Sigma]$ satisfies 
\begin{equation}\label{eq_risque_minimax_inverse_random}
 C_1\frac{k}{n}\log\left(\frac{ep}{k}\right)\leq \mathcal{RI}_R[k]\leq C_2\frac{k}{n}\log\left(\frac{ep}{k}\right)\ .
\end{equation}
\end{prte}

We observe that  $\mathcal{RI}_R[k]$ and 
$\mathcal{RI}_F[k]$ behave similarly in a non-ultra-high dimensional setting.

\begin{remark}{\bf [Ultra-high dimensional case]}
 Proposition \ref{prte_minimax_inverse_random} does not allow to derive the order of magnitude of $\mathcal{RI}_R[k]$ in an ultra-high dimensional setting. While the upper bound in (\ref{minoration_inverse_design_random_rask}) is blowing up, the lower bound remains as small as $k\log(p/k)/n$. Nevertheless, we know from Proposition \ref{prte_prediction_minoration_minimax} that 
\begin{equation*}
 \mathcal{RI}_R[k,I_p]=\mathcal{R}_R[k,I_p]\geq  C_1
\frac{k\log\left(ep/k\right)}{n}\exp\left[C_2\frac{k}{n}\log\left(\frac{ep}{k}\right)\right]\ . 
\end{equation*}
This suggests that $\mathcal{RI}_R[k]$ is blowing up in an ultra-high dimensional setting but the problem remains open.
\end{remark}

In the next proposition, we state the counterpart of Proposition \ref{prte_majoration_probleme_inverse_adaptation} in the random design case.
\begin{prte}[Adaptation to the sparsity and the variance]\label{prte_majoration_probleme_inverse_adaptation_random} As in the prediction case, we consider the estimator $\widetilde{\theta}^V$ (defined in Section \ref{section_theta_V}). Assume that $p\geq 2n$. For any covariance $\Sigma$, any $\sigma>0$, any $1\leq k\leq \lfloor (n-1)/12\rfloor$, and any $\theta_0\in\Theta[k,p]$, we have  
\begin{equation}\label{eq_prte_majoration_probleme_inverse_adaptation_random}
 \frac{\|\widetilde{\theta}^V-\theta_0\|_p^2}{\sigma^2}\leq C_1 \frac{k}{n\varPhi_{3k,-}(\sqrt{\Sigma})}\log\left(\frac{ep}{k}\right)\exp\left[C_2\frac{k}{n}\log\left(\frac{ep}{k}\right)\right]\ ,
\end{equation}
with probability larger than $1-e^{-n}-C/p$.
\end{prte}

\subsection{Consequences on support estimation}

We deduce from the minimax lower bounds for the inverse problem $(\bf{P_3})$
some consequences for the support estimation problem $(\bf{P_4})$ in a
ultra-high dimensional setting. The case $k[1+\log(p/k)]$ small compared to $n$ has been studied in Wainwright \cite{wain_minimax2}.

\begin{defi}\label{definition_hypercube}
 For any $\rho>0$ and any $k\leq p$, the set
$\mathcal{C}_k^p(\rho)$ is made of all vectors $\theta$ in $\Theta[k,p]$ such that
$\theta$
contains exactly $k$ non-zero coefficients that are all equal to 
$\rho/\sqrt{k}$.
\end{defi}

In a non-ultra high dimensional setting, Wainwright \cite{wain_minimax} has proved, that under suitable conditions on a design ${\bf X}\in \mathcal{D}_{n,p}$, it is possible to recover the support of any vector $\theta_0$ that belong to $\mathcal{C}_k^p(\rho)$ with $\rho$ of order of $\sqrt{k\log(p)/n}\sigma$. Here, we prove that $\rho$ has to be much larger in an ultra-high dimensional setting.

\begin{prte}{\bf [Support recovery
is almost impossible]}\label{prte_minoration_support_estimation}
For any $\rho^2\leq C_1/n\left(\frac{ep}{k}\right)^{C_2k/n}$ and any $k\leq n\wedge p/2$, we have 
\begin{eqnarray*}
 \inf_{{\bf
X}\in\mathcal{D}_{n,p}}\inf_{\hat{m}}\sup_{\theta_0\in\mathcal{C}
_k^p(\rho)}\mathbb{P}_{\theta_0,1}\left[\hat{m}\neq
\mathrm{supp}(\theta_0)\right]\geq 1/(2e+1)\ .
\end{eqnarray*}
\end{prte}

For any design ${\bf
X}\in\mathcal{D}_{n,p}$
it is not possible to recover the support of $\theta_0$ with high probability,
unless $\theta_0$ satisfies:
$$\frac{\|\theta_0\|_p^2}{\sigma^2}\geq C_1/n \left(\frac{p}{k}\right)^{C_2k/n}\
 .$$ 
This quantity is blowing up in an ultra-high dimensional setting and it can be much larger than the usual $k\log(p)/n$ that can be achieved in a non-ultra high dimensional setting. \\

As it is almost impossible to estimate the support of $\theta_0$ in an ultra-high
dimensional setting, we may aim to an easier objective. Can we choose a
subset $\widehat{M}$ of $\{1,\ldots, p\}$ of size $p_0\leq p$ that contains the
support of $\theta_0$
with	
high probability? This would allow to reduce the dimension of the problem from
$p$ to $p_0$. Dimension reductions techniques are popular for
analyzing high dimensional problems.
We study here to what extent dimension reduction is a realistic objective:
 how large should be the non-zero components of
$\theta_0$? How small can we choose $p_0$?

\begin{prte}\label{prte_minoratoin_subset_estimation}
Consider a Gaussian design regression with  $\Sigma=I_p$
and 
$\sigma^2=1$.
We assume that $p\geq k^3\vee C$ and $n\geq C$. Set 
$$\rho^2=
C\frac{k}{n}\log\left(\frac{ep}{k}\right)\exp\left[C_2\frac{k}{n}\log\left(\frac{ep}{k} \right)\right]\ .$$
There exists a universal constant $0<\delta<1$ such that for any measurable
subset
$\widehat{M}$ of $\{1,\ldots,p\}$ of size
$p_0\leq p^\delta$, we have
\begin{eqnarray}\label{equation_reduction}
\sup_{\theta_0\in
\mathcal{C}_k^p(\rho)}\mathbb{P}_{\theta_0,1}\left[\mathrm{supp}
(\theta_0) \nsubseteq \widehat{M}
\right]\geq 1/8\ .
\end{eqnarray}
\end{prte}

In an ultra-high dimensional setting, it is therefore not possible to
reduce the dimension of the problem to $p^\delta$ unless the  square norm of
$\theta_0$ is of order $\exp[Ck/n\log(p)]\sigma^2$. In (\ref{equation_reduction}), the
number $1/8$ is of no particular significance. It can be replaced by any
constant $c\in(0,1)$ if we take an asymptotic point of view
($(k,p,n)\rightarrow \infty$).

\begin{remark}
 In Proposition \ref{prte_minoratoin_subset_estimation}, we have taken the maximal risk points of view. If we put an uniform prior $\pi$ on $\mathcal{C}_k^p(\rho)$, it is possible to replace (\ref{equation_reduction}) by 
$$ \pi\left[\mathbb{P}_{\theta_0,1}\left\{\mathrm{supp}
(\theta_0) \nsubseteq \widehat{M}
\right\}\right]\geq C\ ,$$
where $C$ is a positive constant.
\end{remark}

\begin{remark}
In order to shed light on the
problem of dimension reduction, let us consider a simple asymptotic example:
$p_n=\exp(n^{\gamma_1})$   and ${k_n=
n^{1-(\gamma_1\wedge 1)+\gamma_2}}$ with $\gamma_1>0$ and $\gamma_2>0$.  If we
assume that 
$\theta_n\in\Theta[k_n,p_n]$ is such that $\|\theta_n\|_p^2\leq
\exp(C n^{\gamma_2+(\gamma_1-1)_+})$, then it is not possible to find a subset
$\widehat{M}_n$ of size
$\exp(\delta n^{\gamma_1})$ that contains the support of $\theta_n$ with
probability going to one, where $\delta$ is defined as in Proposition
\ref{prte_minoratoin_subset_estimation}. Consequently, we still have
to keep at least $\exp(\delta n^{\gamma_1})$ variables after the process of
dimension reduction if we do not want to forget relevant variables!
\end{remark}

\section{What is an ultra-high dimensional
problem?}\label{section_simulations}

Until now, we have stated that a problem is ultra-high dimensional when
$k\log(p/k)$
is large compared to $n$. It has been proved that in such a setting, estimation of $\theta_0$, support estimation and even dimension reduction become almost impossible. In this section, we numerically illustrate this phase transition phenomenon. This allows us to quantify on specific examples how large should be $k\log(p/k)/n$ for the phase transition to occur.\\

\noindent
{\bf First simulation setting}. Following the example described in the
introduction, we consider a
Gaussian design linear regression model with $p=5000$ and $p=200$, $n=50$,
$\Sigma=I_p$, and
$\sigma=1$. 
We set the number of non zero components $k$ ranging from 1 to 15.  $k$ being
fixed, we
take $\theta_0$ such that $(\theta_0)_1=\ldots=(\theta_0)_k=4\sqrt{\log(p)/n}\approx
1.30$ (resp. 1.65) for $p=200$ (resp. $p=5000$)
and $(\theta_0)_{k+1}=\ldots =(\theta_0)_p=0$. As a consequence, we have
$\|\theta_0\|^2=16k\log(p)/n$. The non-zero coefficients of $\theta_0$ are
chosen large enough so that the support of $\theta_0$ is recoverable 
when the problem is not ultra-high dimensional. Each experiment is
repeated $N=100$ times.\\

\noindent 
{\bf Dimension reduction procedures}. We apply the SIS method~\cite{fan_ultra}
to reduce the
dimension to a set $\widehat{M}^S$ of size $p_0=50$. We then compute the Power
of the procedure,
$$\mathrm{Power}:= \frac{\mathrm{Card}[\widehat{M}^S\cap \{1,\ldots,k\}]}{k}\
.$$
The power measures whether the dimension reduction has been
performed efficiently.

We also compute the regularization path of the Lasso using the LARS~\cite{lars}
algorithm. Before applying the Lasso, each column of ${\bf X}$ is normalized.
We consider the set $\widehat{M}^L$ made of the  $p_0$ covariates occurring
first in the regularization
path.  We do not argue that SIS and the
Lasso are the best methods here. We have chosen them because they are classical
and easy to implement.\\

\begin{figure}[htbp]
\begin{center}
{\includegraphics[height=8cm,angle=-90]{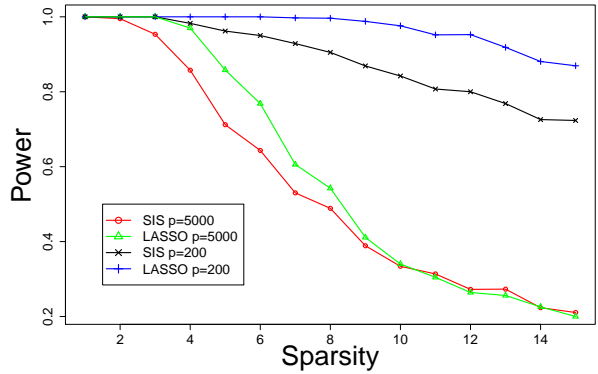}}
\end{center}
\caption{\label{figure1} Power of the dimension reduction
procedures (SIS and Lasso) as a function of $k$.}
\end{figure}

\noindent 
{\bf Results}. The results are presented on Figure \ref{figure1}. When $k$ is
small, the dimension reduction problem is not ultra-high dimensional and the
Lasso and the SIS methods keep all the relevant covariates.
For large $k$, the both methods miss some of the relevant
covariates.  For $p=5000$, there is a clear decrease in the power
beyond $k=4$. For $p=5000$ and $k=8$, both methods only have a power close to
0.5. In expectation, only four covariates belong to the sets $\widehat{M}^S$
and $\widehat{M}^L$ of size 50. For $p=200$, there is not a so clear transition,
but the power decreases slowly for $k>8$. If there was no elbow effect in the minimax risk of estimation, then it would still be possible to recover the support of $\theta_0$ with high probability. Indeed, each non-zero component of $\theta_0$ is larger than $4\sqrt{\log(p)/n}$ which is detectable in a reasonable setting (see e.g.~\cite{wain_minimax}). For instance, for $k=6$ and $p=5000$, $\|\theta_0\|_p^2/\sigma^2=16k\log(p)/n\approx 16.4$. Here, the elbow effect implies that even for a huge signal over noise ratio, it is impossible to reduce the dimension of the problem without forgetting relevant variables.
\\

\noindent 
{\bf Second simulation setting.} We still take $p=5000$, $n=50$,
$\Sigma=I_p$, $\sigma=1$, and $k$ ranging from 1 to 5. $k$ being fixed,
we take $\theta_0$ such that $(\theta_0)_1=\ldots=(\theta_0)_k=u\sqrt{\log(p)/n}$
and $(\theta_0)_{k+1}=\ldots =(\theta_0)_p=0$. Relying on  $N=100$ experiments, we
estimate $u^*_{k}$ the smallest $u$ such that $\widehat{M}^L$ has a power larger
than $0.9$. $u^*_k$ corresponds (up to the renormalization $\sqrt{\log(p)/n}$)
to the minimal intensity of the signal so that the dimension reduction
method does not forget relevant covariates.

\begin{figure}[htbp]
\begin{center}
{\includegraphics[height=7cm,angle=-90]{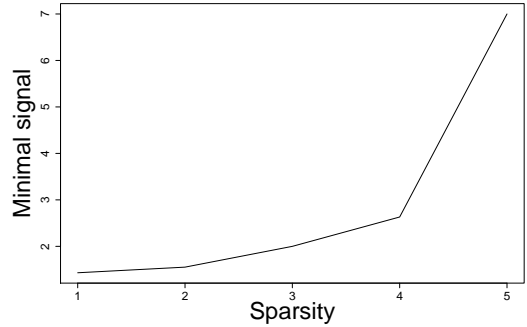}}
\end{center}
\caption{Minimal signal $u^*_k$ as a function of $k$.}\label{figure2} 
\end{figure}
 
~\\\noindent 
{\bf Results.}  The results are presented on Figure \ref{figure2}. For small
$k$, $u^{*}_k$ remains close to $\sqrt{2}$. 
In contrast, we observe that $u^*_k$ blows up at $k=5$. We have not depicted
$u^*_6$, but we have $u^*_6\geq 100$. These two simulation studies confirm that when $k$ becomes large (in
comparison to $p$ and $n$),
the dimension reduction problem becomes extremely
difficult.

\begin{remark}[{\bf Rule of thumb}]
From these simulations and from other theoretical
arguments~(e.g. \cite{giraud08,dota,wain_minimax2}), we derive a simple rule of thumb. We say that
a
problem is ultra-high dimensional if
\begin{equation}
\text{\fbox{ $\frac{k\log(p/k)}{n}\geq 1/2$.}}
\end{equation}\label{critere}
For $p=5000$ and $n=50$, this corresponds to $k\geq 4$. Setting $p=200$ and
$n=50$ yields $k\geq 8$. In practice, we do not know $k$ in advance. Nevertheless, this criterion (\ref{critere}) helps us to know what is the largest sparsity index such that the statistical problem remains reasonably difficult in the minimax sense.	
\end{remark}

\section{Discussion}\label{section_conclusion}

As stated in Sections \ref{section_linear_testing}--\ref{section_inverse}, the
behaviors of the minimax separation distances and of the minimax risks become
really different in an ultra-high dimensional setting. Apart from the test
problem (${\bf P_1}$) with known variance and the problem  of prediction~(${\bf
P_2}$) with fixed
design, all the other separations distances and minimax risks blow up when 
$k\log(p/k)$ becomes larger than $n$.

This elbow effect has important practical implications: there is no hope of
selecting the 
relevant covariates in an ultra-high dimensional setting, except if 
signal over noise ratio is exponentially large. Moreover, even dimension
reduction techniques cannot work well in such a setting.\\

In linear testing (${\bf P_1}$), we have proved that the optimal separation
distances highly depend on the knowledge of the variance.
Most of the testing procedures in the
literature rely on the knowledge of
$\sigma^2$. Some specific work is therefore needed to derive fast
and efficient procedures under unknown variance (but see~\cite{Ingster10} for a procedure in a specific situation).\\


We have
not discussed so far the problem of variance estimation. From the 
minimax lower bounds of testing, we deduce the following lower bound.
\begin{prte}\label{prte_minoration_variance} Assume that $p\geq n\geq C$.
For any $1\leq k\leq p^{1/3}$, there exist designs ${\bf X}$ such that 
\begin{equation*}
\inf_{\widehat{\sigma}} \sup_{\sigma>0,\ \theta_0\in\Theta[k,p]}\mathbb{E}_{\theta_0,\sigma}\left[\left|\frac{\widehat{\sigma}^2}{\sigma^2}-\frac{\sigma^2}{\widehat{\sigma}^2}\right|\right]\geq C_1\frac{k}{n}\log\left(\frac{p}{k}\right)\exp\left[C_2\frac{k}{n}
\log\left(\frac{p}{k}\right)\right ]\ .
\end{equation*}
\end{prte}
As a consequence, the problem of variance estimation becomes extremely difficult in an ultra-high dimensional setting. \\

In Propositions \ref{prte_minimax_prediction_fixed_design} and
\ref{prte_minoration_inverse_fixe}, we have provided minimax lower bounds
for (${\bf P_2}$)  and (${\bf P_3}$) over $\Theta[k,p]$ for arbitrary designs
${\bf
X}$. Our corresponding upper bounds match these lower bounds when the
restricted eigenvalues of ${\bf X}^T{\bf X}$ are close to each other. However, these
bounds do not agree anymore when these restricted eigenvalues are away from each other.
 Deriving the exact
dependency of the minimax risks on ${\bf X}$ would require sharper lower bounds
and the analysis of
new estimation procedures.\\

Our minimax results use the  Gaussianity of the noise
$\boldsymbol{\epsilon}$ and the Gaussianity of the design ${\bf X}$ in the
random design setting.
In an ultra-high dimensional setting, the minimax upper bounds do not seem to be
robust with respect to the Gaussianity. In smaller dimensions
($k[1+\log(p/k)]<n$), the Gaussian distribution of the design
is less critical. For instance, consider a design ${\bf X}$ where all the
components are independent and follow a
subgaussian distribution. By a result of Rudelson and 
Vershynin~\cite{rudelson}, the
restricted eigenvalues of ${\bf X}^T{\bf X}$ remain away from $0$ with high
probability. Consequently, some of the minimax bounds should still hold for
subgaussian designs.
Nevertheless, the derivation of sharp minimax bounds for non-Gaussian designs
and noises remains an open problem

\section{Proofs of the minimax lower bounds}\label{section_proof_lowerbound}

Some propositions contain both minimax lower bounds and upper bounds. This section is devoted to the proof of the main lower bounds, while the upper bounds are proved in Appendix B in \cite{technical}. In order to keep our notations as short as possible, we set 
$$\eta=2(1-\alpha-\delta)\ .$$
We also note $\|.\|_{TV}$ for the total variation norm. For any subset
$\mathcal{T}\subset\mathbb{R}^p$, $\alpha\in(0,1)$, covariance
matrix $\Sigma$, and  any variance $\sigma^2$,  we denote $\beta^R_{\Sigma,
\sigma, \alpha}(\mathcal{T})$ the quantity
$$\beta^R_{\Sigma, \sigma, \alpha}(\mathcal{T}):=
\inf_{\Phi_{\alpha}}\sup_{\theta_0\in\mathcal{T}}
\mathbb{P}_{\sigma\theta_0,\sigma}[\Phi_{\alpha}=0]\ ,$$
the infimum being taken over all tests $\Phi_{\alpha}$ satisfying
$\mathbb{P}_{0_p,\sigma}[\Phi_{\alpha}=0]\leq \alpha$. Its counterpart for unknown variance is defined by
$$\beta^R_{\Sigma, \alpha}(\mathcal{T}):=
\inf_{\Phi_{\alpha}}\sup_{\sigma>0,\ \theta_0\in\mathcal{T}}
\mathbb{P}_{\sigma\theta_0,\sigma}[\Phi_{\alpha}=0]\ ,$$
the infimum being taken over all tests $\Phi_{\alpha}$ satisfying
$\sup_{\sigma>0}\mathbb{P}_{0_p,\sigma}[\Phi_{\alpha}=0]\leq \alpha$. 
Similarly, we define
$\beta^F_{{\bf X},\sigma,\alpha}(\mathcal{T})$ for fixed design and $\beta^F_{{\bf
X},\alpha}(\mathcal{T})$ for fixed design and unknown variance.

Most of the minimax lower bounds in this paper are based on an approach which
goes back to Ingster~\cite{ingster93a,ingster93b,ingster93c}. The following
lemma encompasses
fixed and random design and known  and unknown variance.
\begin{lemma}\label{lemma_le_cam_method} 
Let $\mathcal{T}$ be a subset of
$\mathbb{R}^p\setminus\{0_p\}$ and let $\sigma$ and $\sigma_0$ be two positive integers. Consider
 $\mu$ a probability
measure on $\sigma\mathcal{T}:=\{\sigma\theta,\ \theta\in\mathcal{T}\}$.  We note
$\mathbb{P}_{\mu,\sigma}=\int_{\sigma\mathcal{T}}\mathbb{P}_{\theta ,\sigma}d\mu$ and
$L_{\mu}= d\mathbb{P}_{\mu,\sigma}/d\mathbb{P}_{0_p,\sigma_0}$. Then,
\begin{eqnarray}
\beta_{\alpha}(\mathcal{T})&\geq &1-
\alpha-\frac{1}{2}\|\mathbb{P}_{\mu,\sigma}-\mathbb{P}_{0_p,\sigma_0}\|_{TV}
. \nonumber\\
&\geq & 1- \alpha -
\frac{1}{2}\left(\mathbb{E}_{0_p,\sigma_0}\left[L^2_{\mu}({\bf
Y},{\bf X})\right]-1\right)^{1/2}\ .
\label{eq_fondamental_lemma}
\end{eqnarray}
Here, $\beta_{\alpha}(\mathcal{T})$ can be replaced by $\beta^F_{{\bf X}, \alpha}(\mathcal{T})$ or $\beta^R_{\Sigma,\alpha}(\mathcal{T})$. If we also have $\sigma=\sigma_0$, then $\beta_{\alpha}(\mathcal{T})$  can be replaced by  $\beta^R_{\Sigma,\sigma_0,\alpha}(\mathcal{T})$ or  $\beta^F_{{\bf X},\sigma_0,\alpha}(\mathcal{T})$.
\end{lemma}
We refer to Baraud~\cite{baraudminimax}~Section 7.1 for a proof and further
explanations in a close framework. The main idea is to find a prior probability
on $\mathcal{T}$ so that the total variation distance between
$\mathbb{P}_{\mu,\sigma}$ and $\mathbb{P}_{0_p,\sigma_0}$ is as large as possible. 
We derive from Lemma \ref{lemma_le_cam_method} that  $\beta_{\alpha}(\mathcal{T})\geq \delta$ if
$\mathbb{E}_{0_p,\sigma_0}[L^2_{\mu}({\bf
Y},{\bf X})]\leq 1+\eta^2$.

\subsection{Proof of the lower  bound (\ref{minoration_test_random_eq_known_variance2}) in Theorem
\ref{thrm_minoration_minimax_variance_conditionnelle}}

\begin{proof}[Proof of Theorem
\ref{thrm_minoration_minimax_variance_conditionnelle}]

By homogeneity, we can assume that $\sigma^2= \var(Y|X)=1$. We first build a
suitable prior probability $\mu_{\rho}$  in order to apply Lemma
\ref{lemma_le_cam_method}.\\

 Let us take a set $\hat{m}$ of size $k$
uniformly in $\mathcal{M}(k,p)$ (defined in Section \ref{section_notation}).
Let
$\xi=(\xi_j)_{1\leq j\leq p}$
be a sequence of independent Rademacher random variables. Consider some
$\rho>0$. Define
$\lambda=\rho/\sqrt{k}$ and consider
$\mu_{\rho}$ the distribution of the random variable
$\theta_{\hat{m},\xi}=\sum_{j\in\hat{m}}\lambda\xi_j e_j$.  $\mathbb{P}_{\mu_{\rho},1}$ stands for the distribution of $({\bf Y},{\bf X})$ with $\theta_0\sim\mu_{\rho}$ and $\sigma=1$.
Here, $(e_j)_{1\leq j\leq p}$ is the orthonormal family of vectors
of 
$\mathbb{R}^{p}$ defined by 
$$ (e_j)_i =1 \text{ if } i=j \text{ and } (e_i)_j=0 \text{ otherwise}. $$

The likelihood ratio 
$L_{\mu_{\rho}}({\bf X},{\bf Y})= \mathbb{P}_{\mu_{\rho},1}/\mathbb{P}_{0_p,1}$ writes
\begin{eqnarray*}
 L_{\mu_{\rho}}({\bf X},{\bf Y}) =
\mathbb{E}_{\xi,\hat{m}}\left[\exp\left(-\frac{\|{\bf Y}-{\bf
X}\theta_{\hat{m},\xi}\|_n^2-\|{\bf Y}\|_n^2}{2}\right)\right]\ ,
\end{eqnarray*}
where $\mathbb{E}_{\xi,\hat{m}}$ stands for the expectation with respect to the distribution of $\xi$ and $\widehat{m}$.

In order to apply Lemma \ref{lemma_le_cam_method}, we need to upper bound the
expectation of $L^2_{\mu_{\rho}}({\bf X},{\bf Y})$.
Let us first take the expectation of $L^2_{\mu_{\rho}}({\bf X},{\bf Y})$ with
respect to ${\bf Y}$.
\begin{eqnarray}
 \lefteqn{\mathbb{E}_{0_p,1}\left[L^2_{\mu_{\rho}}({\bf X},{\bf Y})\right]
}\nonumber\\&=& 2^{-2k}\binom{p}{k}^{-2}\sum_{m_1,m_2,\xi^{(1)},\xi^{(2)}}
\mathbb{E}_{0_p,1}\left[e^{-\left(\|{\bf
X}\theta_{m_1,\xi^{(1)}}\|_n^2+\|{\bf
X}\theta_{m_2,\xi^{(2)}}\|_n^2\right)/2+\left\langle {\bf Y},{\bf
X}\left(\theta_{m_1,\xi^{(1)}}+ \theta_{m_2,\xi^{(2)}}\right)\right\rangle_n
}\right]\nonumber\\
& = &
\mathbb{E}_{\bf
X}\left[\mathbb{E}_{\hat{m}_1,\hat{m}_2,\xi^{(1)},\xi^{(2)}}\left\{\exp\left(\langle{ \bf
X}\theta_{\hat{m}_1,\xi^{(1)}},{\bf X}\theta_{\hat{m}_2,\xi^{(2)}}\rangle_n\right)\right\}\right]\
,\label{expression_lmu2}
\end{eqnarray}
where $\mathbb{E}_{\bf X}$ stands for the expectation with respect to ${\bf X}$ while $\mathbb{E}_{\hat{m}_1,\hat{m}_2,\xi^{(1)},\xi^{(2)}}$ refers to the expectation with respect to the independent variables $\xi^{(1)}$, $\xi^{(2)}$, $m_1$ and $m_2$.

\begin{lemma}\label{lemma_majoration_design_standard}
If we assume that $\rho^2\leq C \left[\frac{k}{n}\log\left(1+\frac{p}{k^2}\right)\wedge \frac{1}{\sqrt{n}}\right]$,
 then we have
\begin{eqnarray*}
\mathbb{E}_{\bf X}\left[\mathbb{E}_{0_p,1}\left\{\left. L^2_{\mu_{\rho}}\left({\bf
Y},{\bf
X}\right)\right|{\bf X}\right\}\right]\leq 1+ \eta^2 \ .
\end{eqnarray*}
\end{lemma}
In this lemma, we have specifically distinguished the integration with respect
to ${\bf X}$ from the integration with respect to ${\bf Y}$. This
will be useful for deriving minimax lower bound in fixed design (Proposition
\ref{prte_minoration_test_fixed}).
Gathering Lemmas \ref{lemma_le_cam_method} and
\ref{lemma_majoration_design_standard}
 allows to derive that
\begin{eqnarray*}
 (\rho_R^*[k,I_p])^2\geq
C\left[\frac{k}{n}\log\left(1+
\frac{p}{k^2}\right)\wedge \frac{1}{\sqrt{n}}\right]\ .
\end{eqnarray*}
This last bound allows to conclude since $p\geq n^2$.

\end{proof}

\begin{proof}[Proof of Lemma \ref{lemma_majoration_design_standard}]
 
Let us fix $m_1$, $m_2$, $\xi^{(1)}$ and $\xi^{(2)}$. First, we shall compute the
expectation ~\\$\mathbb{E}_{\bf X}[\exp(\langle {\bf X}\theta_{m_1,\xi^{(1)}},{\bf
X}\theta_{m_2,\xi^{(2)}}\rangle_n)]$. \\

Let us  decompose the set $m_1\cup m_2$ into 
four sets (which possibly are empty): $m_1\setminus m_2$, $m_2\setminus m_1$,
$m_3$, and $m_4$, where $m_3$ and $m_4$ are defined by $m_3  :=  \{j\in m_1\cap m_2|\xi^{(1)}_j =\xi^{(2)}_j \}$ and $m_4  :=  \{j\in m_1\cap m_2|\xi^{(1)}_j =-\xi^{(2)}_j \}$ .
For the sake of simplicity, we reorder the elements of 
$m_1\cup m_2$ from $1$ to $|m_1\cup m_2|$ such that the first elements 
belong to $m_1 \setminus m_2$, then to $m_2\setminus 
m_1$ and so on.
\begin{eqnarray*}
 \mathbb{E}_{\bf X}\left[\exp\left(\langle {\bf X}\theta_{m_1,\xi^{(1)}},{\bf
X}\theta_{m_2,\xi^{(2)}}\rangle_n\right)\right] &=&  
\left[\int_{\mathbb{R}^p}(2\pi)^{-p/2}\exp\left(-\sum_{i=1}^p t_i^2/2+
\sum_{1\leq i,j\leq
p}[\theta_{m_1,\xi^{(1)}}]_{i}[\theta_{m_2,\xi^{(2)}}]_{j}t_it_j\right)\prod_{i=1}^p
dt_i\right]^n\\ 
& = & \left|I_{|m_1\cup m_2|} - \lambda^2C\right|^{-n/2}\ ,
\end{eqnarray*}
where $I_{|m_1\cup m_2|}$ is the identity matrix of size $|m_1\cup m_2|$ and
$C$ is block symmetric matrix of size  $|m_1\cup m_2|$ defined by
\begin{eqnarray*} 
C :=  
\left[\begin{array}{cccc}  
1&  0  & 1& 1
\\ 0 &1&1&1
\\  1&1 &2 & 0  
\\ 1& 1 & 0 &
-2
\end{array}\right] \, .
\end{eqnarray*} 
Each block corresponds to one of the four previously defined subsets of $m_1\cup
m_2$ 
(i.e. $m_1 \setminus m_2$, $m_2 \setminus m_1$, $m_3$, and $m_4$). The matrix
$C$ 
is of rank at most four. Hence, $I_{|m_1\cup m_2|} - \lambda^2C$ has
the same determinant as the matrix $D$ of size $4$ defined by:
\begin{eqnarray*} 
D :=  
\left[\begin{array}{cccc}  
1- \frac{\lambda^2}{n}|m_1\setminus m_2|&  0  & - \frac{\lambda^2}{n}|m_3|& -
\frac{\lambda^2}{n}|m_4|
\\ 0 &1- \frac{\lambda^2}{n}|m_2\setminus m_1|&- \frac{\lambda^2}{n}|m_3|&-
\frac{\lambda^2}{n}|m_4|
\\  - \frac{\lambda^2}{n}|m_1\setminus m_2|&- \frac{\lambda^2}{n}|m_2\setminus
m_1| &1-2 \frac{\lambda^2}{n}|m_3| & 0  
\\ - \frac{\lambda^2}{n}|m_1\setminus m_2|& - \frac{\lambda^2}{n}|m_2\setminus
m_1| & 0 &
1+2\frac{\lambda^2}{n}|m_4|
\end{array}\right] \, .
\end{eqnarray*} 
After some computations, we lower bound the determinant of $D$
\begin{eqnarray*}
|D| &\geq &1- 2(2|m_3|-|m_1\cap m_2|)\lambda^2- 
8\rho^4\  \ .
\end{eqnarray*}
From now on, we assume that $\rho^2\leq 1/20$ so that $|D|\geq 1/2$. Hence, we
get
\begin{eqnarray}
 \mathbb{E}_{\bf X}[\exp(\langle {\bf X}\theta_{m_1,\xi^{(1)}},{\bf
X}\theta_{m_2,\xi^{(2)}}\rangle_n)]&\leq& \left[1- 2(2|m_3|-|m_1\cap
m_2|)\lambda^2- 
8\rho^4\right]^{-n/2}\nonumber\\
 &\leq & \exp\left(8n\rho^4\right)\exp\left[2n\lambda^2(2|m_3|-|m_1\cap
m_2|)\right]\label{expression_vraisemblance}\ . 
\end{eqnarray}
Then, we take the expectation with respect to $\xi^{(1)}$, 
$\xi^{(2)}$, $m_1$ and $m_2$. When $m_1$ and $m_2$ are fixed the expression
(\ref{expression_vraisemblance}) depends on $\xi^{(1)}$ and $\xi^{(2)}$ only through
the cardinality of $m_3$. As $\xi^{(1)}$ and $\xi^{(2)}$ follow independent
Rademacher distributions, the random variable $2|m_3|-|m_1\cap m_2|$ follows the
distribution of $Z$, a sum of $|m_1\cap m_2|$ independent Rademacher variables
and
\begin{eqnarray}\label{majoration_principale_risque}
\mathbb{E}_{{\bf X}}\left[\mathbb{E}_{0_p,1}\left\{\left. L^2_{\mu_{\rho}}\left({\bf Y},{\bf
X}\right)\right|{\bf X}\right\}\right]\leq
\exp\left(8n\rho^4\right)\mathbb{E}_{Z}\left[
\exp\left(2n\lambda^2Z\right)\right] \ ,
\end{eqnarray}
where $\mathbb{E}_Z$ stands for the expectation with respect to $Z$.
We now proceed as in the proof of Theorem 1 in  Baraud~\cite{baraudminimax} in
order to upper bound the term
\begin{eqnarray*}
\mathbb{E}_{Z}\left[
\exp\left(2n\lambda^2Z\right)\right] &= &
\binom{p}{k}^{-2}\sum_{m_1,m_2\in\mathcal{M}(k,p)}\cosh\left(2n\lambda^2\right)^
{
|m_1\cap m_2|}\ . 
\end{eqnarray*}
Following Baraud's arguments, we get that $\mathbb{E}_Z\left[
\exp\left(2n\lambda^2Z\right)\right]\leq \sqrt{1+\eta^2}$ when
\begin{eqnarray*}
 \rho^2\leq C \frac{k}{n}\log\left(1+ \frac{p}{k^2}\vee
\sqrt{\frac{p}{k^2}}\right)\ .
\end{eqnarray*}
Moreover, we have $\exp(8\rho^4n)\leq \sqrt{1+\eta^2}$ as soon as
$\rho^2\leq C/\sqrt{n}$\ since $\eta \geq 0.94$. Gathering these observations with
(\ref{majoration_principale_risque}), we conclude that
$\mathbb{E}_{\bf X}\left[\mathbb{E}_{0_p,1}\{\left. L^2_{\mu_{\rho}}\left({\bf Y},{\bf
X}\right)\right| {\bf X}\}\right]\leq 1+\eta^2$ as soon as
\begin{eqnarray*}
 \rho^2\leq C \left[\frac{k}{n}\log\left(1+ \frac{p}{k^2}\vee
\sqrt{\frac{p}{k^2}}\right)\wedge \frac{1}{\sqrt{n}}\right]\ .
\end{eqnarray*}
\end{proof}

\subsection{Proof of the lower bound (\ref{borne_minoration_test}) in Theorem \ref{thrm_minimax_testing}}

\begin{proof}[Proof of (\ref{borne_minoration_test}) in Theorem \ref{thrm_minimax_testing}]

 Consider the Condition
 $${\bf
(A.1)}\hspace{6.5cm}\frac{k}{n}\log\left(\frac{p}{e^3k^2}\right)\geq 2\
.\hspace{6.5cm}$$
We deduce Theorem \ref{thrm_minimax_testing} from the following result.
\begin{lemma}\label{lemma_minimax_testing}
Suppose that $\alpha+\delta\leq 53\%$. We have
\begin{eqnarray}\label{eq_minoration_erreur}
\beta^R_{I_p,\alpha}\left(\left\{\theta_0\in\Theta[k,p], \|\theta_0\|_p^2=\rho^2\right\} \right)\geq \delta\
,
\end{eqnarray}
for any $\rho^2>0$ such that 
\begin{eqnarray}\label{minoration_test_petitedimension}
\rho^2\leq 
\frac{k}{2n}\log\left(1+\frac{p}{k^2} \right)\ .
\end{eqnarray}
If we assume that Condition ({\bf A.1}) holds, (\ref{eq_minoration_erreur}) holds for any $\rho>0$
such that 
\begin{eqnarray}\label{minoration_test}
\rho^2\leq  -1
+\left(\frac{p}{2ek}\right)^{\frac{k}{n}}(4k)^{-2/n}\ .
\end{eqnarray}
\end{lemma}

If $p\geq k^{3}\vee C$ and  $k\log(p)/n\geq C_1$ with $C$ and $C_1$ large
enough, then Assumption $({\bf A.1})$ is satisfied. For $C$ large enough,
the quantity $k\log(p)/\log(k)$ is large enough so that the lower bound
(\ref{minoration_test}) satisfies  
\begin{eqnarray*}
  -1
+\left(\frac{p}{2ek}\right)^{\frac{k}{n}}(4k)^{-2/n}&\geq&
-1+\exp\left[C\frac{k}{n}\log\left(p\right)\right]\geq
C_1\frac{k}{n}\log\left(p\right)
\exp\left[C_2\frac{k}{n}\log\left(p\right)\right]\ .
\end{eqnarray*}
Let us now assume that $p\geq k^{3}\vee C$ and $k\log(p)/n\leq C_1$ where
$C_1$ has been previously fixed. Then,
the first
lower bound (\ref{minoration_test_petitedimension}) satisfies:
\begin{eqnarray*}
 \frac{k}{2n}\log\left(1+\frac{p}{k^2}\right)\geq
C_1\frac{k}{n}\log\left(p\right)
\exp\left[C_2\frac{k}{n}\log\left(p\right)\right]\ .
\end{eqnarray*}
Gathering the two previous lower bounds with Lemma \ref{lemma_minimax_testing}
allows to conclude.
\end{proof}

\begin{proof}[Proof of Lemma \ref{lemma_minimax_testing}]
Consider some $\rho>0$. To apply Lemma \ref{lemma_le_cam_method}, we first have
to define a suitable prior 
$\mu_{\rho}$ on $\theta_0$ and a suitable $\sigma^2$. More specifically, we set $\sigma^2=(1+\rho^2)^{-1}$ and the
distribution $\mu_{\rho}$ is supported by
 $\Theta[k,p,\rho]$ defined by
\begin{eqnarray*}
 \Theta[k,p,\rho]&:=& \left\{\theta_0\in\Theta[k,p]\ ,\
\|\theta_0\|_p^2
=\frac{\rho^2}{1+\rho^2}\right\}\ .
\end{eqnarray*}

 Let $\hat{m}$ be a random variable
uniformly 
distributed over $\mathcal{M}(k,p)$. Let $\mu_{\rho}$ be the 
distribution of the random variable 
$\widehat{\theta}= \sum_{j\in\hat{m}} \lambda e_j$
where 
$$  \lambda^2 := \frac{\rho^2}{k(1+\rho^2)}\ ,$$
and where $(e_j)_{1\leq j\leq p}$ is the orthonormal family of vectors
of 
$\mathbb{R}^{p}$ defined by $ (e_j)_i =1$ if  $i=j$ and 
$(e_i)_j=0$ otherwise. 
By  Lemma \ref{lemma_le_cam_method}, we only have to prove under conditions (\ref{minoration_test_petitedimension}) or (\ref{minoration_test}) with (${\bf A.1}$), we have
\begin{equation}\label{majoration_moment22}
 \mathbb{E}_{0_p,1}(L_{\mu_{\rho}}^2({\bf Y},{\bf X}))\leq
1+\eta^2\ .
\end{equation}
Observe here that we use a variance $1$ for ${\bf H_0}$ and a variance
$1-\|\theta_0\|_p^2$ for the hypothesis ${\bf H_1}$. Using these two different
variances allows us to take advantage of the fact that we work under unknown
variance.

As a specific case of  \cite{villers1} Eq.(8.5), we have
\begin{eqnarray*} 
\mathbb{E}_{0_p,1}(L_{\mu_{\rho}}^2({\bf Y},{\bf X})) & = &
\binom{p}{k}^{-2}\sum_{m_1,m_2\in\mathcal{M}(k,p)}\left(1-\frac{\rho^2|m_1\cap m_2|}{(1+\rho^2)k }\right)^{-n}\\
&= & \mathbb{E}_{Z}\left[\left(1-\frac{\rho^2Z}{(1+\rho^2)k }\right)^{-n}\right]\ ,
\end{eqnarray*} 
where $Z$ follows an hypergeometric distribution with parameters $p$, $k$, and $k/p$.
We know from 
Aldous (p.173) \cite{aldous85} that $Z$ follows the same distribution as 
the random variable $\mathbb{E}(W|\mathcal{B}_p)$ where $W$ is a binomial random 
variable of parameters $k$, $k/p$ and $\mathcal{B}_p$ some suitable 
$\sigma$-algebra. By a convexity argument, we get
\begin{eqnarray}\label{majoration_moment2}
 \mathbb{E}_{Z}\left[\left(1-\frac{\rho^2Z}{(1+\rho^2)k}
\right)^{-n}\right]\leq
\mathbb{E}_{W}\left[\left(1-\frac{\rho^2W}{(1+\rho^2)k}
\right)^{-n}\right]\ .
\end{eqnarray}
Hence, we only need to upper bound the expectation of the second random
variable. 

\noindent 
{\bf CASE 1}: Proof of Equation (\ref{minoration_test_petitedimension}).
Since $\log(1+x)\leq x$ and since $W\leq k$, we have 
\begin{eqnarray*}
\mathbb{E}_{W}\left[\left(1-\frac{\rho^2W}{(1+\rho^2)k}
\right)^{-n}\right]&\leq& \mathbb{E}_{W}\left[\exp\left(\frac{n\rho^2W/k}{1+\rho^2-\rho^2W/k}\right)\right]\leq \mathbb{E}_{W}\left[\exp\left(n\rho^2W/k\right)\right]\\
 & \leq & \left[1+ \frac{k}{p}\left(e^{n\rho^2/k}-1\right)\right]^k\leq \exp\left[\frac{k^2}{p}(e^{n\rho^2/k}-1)\right]\ .
\end{eqnarray*}
As a consequence, the condition (\ref{majoration_moment22}) holds if $\rho^2\leq \frac{k}{n}\log\left[1+\frac{p}{k^2}\log(1+\eta^2)\right]$. Observe that $\log(1+\eta^2)\geq 0.6$. Since $\log(1+ux)\geq u\log(1+x)$ for any $0<u<1$ and any $x>0$, the last condition is enforced by $\rho^2\leq \frac{k}{2n}\log\left[1+\frac{p}{k^2}\right]$.\\

\noindent
{\bf CASE 2}: Proof of Equation (\ref{minoration_test}). Here, we bound (\ref{majoration_moment2}) under condition $({\bf A.1})$.
We have
\begin{eqnarray*}
\mathbb{E}_{W}\left[\left(1-\frac{\rho^2W}{(1+\rho^2)k}\right)^{-n}
-1\right]&\leq&
\sum_{i=1}^{k}\mathbb{P}\left[W\geq
i\right]\left(1-\frac{\rho^2 i}{(1+\rho^2)k} \right)^{-n} \ .
\end{eqnarray*}
Since we need to ensure that
$\mathbb{E}_{W}[\{1-\rho^2 W/((1+\rho^2)k)\}^{-n}-1]\leq \eta^2$, it is
sufficient to prove that
\begin{eqnarray}\label{majoration_partie1}
\mathbb{P}\left[W\geq
i\right]\left(1- \frac{i}{k}\right)^{-n} &\leq&
\frac{\eta^2 i^{-i}}{4} \text{ for any } 1\leq i\leq \lfloor k/2\rfloor\ , \\
\mathbb{P}\left[W\geq
i\right]\left(1- \frac{\rho^2i}{(1+\rho^2)k}\right)^{-n}& \leq &
\frac{\eta^2}{2k} \text{
for any } \lfloor k/2\rfloor +1\leq i\leq  k \ . \label{majoration_partie2}
\end{eqnarray}

In order to prove these bounds, we shall use a deviation inequality of the
random variable $W/k$.

\begin{lemma}\label{lemma_concentration}
For any $k\geq 1$, $0<x\leq 1$, it holds that  
\begin{eqnarray}\label{inegalite_concentration}
 \mathbb{P}\left[\frac{W}{k}\geq
x\right]\leq \left[\left(\frac{k}{px}\right)^x\frac{1}{(1-x)^{1-x}}\right]^k\ .
\end{eqnarray}

\end{lemma}

\noindent
{\bf FACT 1.} 
For  any $1\leq i\leq \lfloor k/2\rfloor$, the upper bounds
(\ref{majoration_partie1}) hold under Condition (${\bf A.1}$).\\

\noindent
{\bf FACT 2.} The upper bound (\ref{majoration_partie2}) holds for any $\lfloor
k/2\rfloor +1\leq i\leq  k$ as soon as 
\begin{eqnarray}\label{conclusion_preuve_minoration_principale}
\rho^2\leq -1 +
\left(\frac{p}{2ek}\right)^{k/n}\left(\frac{\eta^2}{2k}\right)^{2/n}\ .
\end{eqnarray}
We derive that under (\ref{conclusion_preuve_minoration_principale}), we have
$\mathbb{E}_{0_p,1}[L_{\mu_{\rho}}^2({\bf Y},{\bf X})]\leq
1+\eta^2\ .$
The fact that $\eta^2\geq 1/2$ allows to conclude.

\end{proof}

\begin{proof}[Proof of FACT 1]

Since $\log(1-x)\geq -x/(1-x)$ for any $0\leq x < 1$, we derive that
$(1-x)^{1-x}\geq e^{-x}$. Gathering this bound with Lemma
\ref{lemma_concentration}, we get a new deviation inequality for $W$.
\begin{eqnarray}\label{concentration_x_petit}
 \mathbb{P}\left[\frac{W}{k}\geq
x\right]\leq \left(\frac{ke}{px}\right)^{xk}\ ,
\end{eqnarray}
for any $x<1$. We apply this bound with $x=i/k$. Then, Inequality
(\ref{majoration_partie1}) holds if
\begin{eqnarray*}
\left(\frac{k^2e}{p}\right)^{i/n}\left(\frac{4}{\eta^2}\right)^{1/n}\leq
1-\frac{i}{k}\ .
\end{eqnarray*}
Taking the logarithm of this expression leads to 
\begin{eqnarray*}
-\frac{i}{n}\log\left(\frac{p}{ek^2}\right)+\frac{1}{n}\log\left(4/\eta^2
\right)+ \frac{i/k}{1-i/k}\leq  0\ ,
\end{eqnarray*}
Since $i$ is constrained to be smaller than $k/2$, we get
\begin{eqnarray*}
-\frac{ik}{n}\log\left(\frac{p}{ek^2}\right)+\frac{k}{n}\log\left(4/\eta^2
\right)+ 2i\leq  0\ .
\end{eqnarray*}
By Assumption ({\bf A.1}), $k/n\log[p/(ek^2)]$ is larger than $2$.
Consequently, the worst case among all $i$ between $1$ and $k/2$ is $i=1$.
Hence, we only need to prove that
\begin{eqnarray*}
\frac{k}{n}\left[\log\left(\frac{p}{k^2}\right)-
\log\left(\frac{4e}{\eta^2}\right)\right]\geq 2\ .
\end{eqnarray*}
Since $\eta$ is larger than $0.94$, $\log(4e/\eta^2)$ is smaller than $3$ and
this last inequality is ensured by Assumption ({\bf A.1}).

\end{proof}

\begin{proof}[Proof of FACT 2]

We consider  here the case  $1/2 < i/k\leq 1$. We derive from
(\ref{concentration_x_petit}) that
\begin{eqnarray*}
\mathbb{P}\left[W\geq
i\right]\leq   
\left(\frac{2ek}{p}\right)^{i}\ .
\end{eqnarray*}
Consequently, we want to ensure that
\begin{eqnarray*}
   \left(\frac{2ek}{p}\right)^{i/n}\left(\frac{2k}{\eta^2}\right)^{1/n}\leq
\left(1- \frac{\rho^2i}{(1+\rho^2)k}\right)\ ,
\end{eqnarray*}
for any $i$ between $\lfloor k/2\rfloor$ and $k$.
For any $x$ and $u$ between $0$ and $1$, $(1-x)^u\leq (1-xu)$. Setting $u=i/k$
and $x=\rho^2/(1+\rho^2)$, we obtain that the last inequality holds if
\begin{eqnarray*}
1-\frac{\rho^2}{1+\rho^2}\geq \sup_{\lfloor k/2\rfloor \leq i\leq k}
\left(\frac{2ek}{p}\right)^{k/n}\left(\frac{2k}{\eta^2}\right)^{k/(in)}
\end{eqnarray*}
Since $2k/\eta^2$ is positive, the largest term in the bound corresponds to
$i=k/2$. Hence, it remains to prove that
\begin{eqnarray*}
\frac{1}{1+\rho^2}\geq
\left(\frac{2ek}{p}\right)^{k/n}\left(\frac{2k}{\eta^2}\right)^{2/n}
\end{eqnarray*}
We conclude that the upper bounds  hold  if 
\begin{eqnarray*}
\rho^2\leq -1 +
\left(\frac{p}{2ek}\right)^{k/n}\left(\frac{\eta^2}{2k}\right)^{2/n}\ .
\end{eqnarray*}

\end{proof}

\begin{proof}[Proof of Lemma \ref{lemma_concentration}]
We prove this deviation inequality using the Laplace transform of $W/k$.
Consider some $x\in (0,1)$ and $\lambda>0$. 
\begin{eqnarray*}
 \log\left[\mathbb{P}\left\{\frac{W}{k}\geq x\right\}\right]&\leq &
-\lambda x+ \log\left[ \mathbb{E}_{W}\left\{\exp(\lambda W/k)\right\}\right]
 \leq  -\lambda
x+k\log\left[1+\frac{k}{p}\left(\exp\left(\frac{\lambda}{k}
\right)-1\right)\right] \ .
\end{eqnarray*}
Deriving with respect to $\lambda$ an upper bound of the last expression leads
to the following choice
$$e^{\lambda^*/k }=\frac{x}{1-x}\left(\frac{p}{k}-1\right) \ .$$
Hence, we get
\begin{eqnarray*}
 \log\left[\mathbb{P}\left\{\frac{W}{k}\geq x\right\}\right]\leq
-kx\log\left[\frac{x}{1-x}\left(\frac{p}{k}-1\right)\right]+k\log\left[
\frac{1-k/p}{1-x}\right]\ .
\end{eqnarray*}
Since we assume  $x<1$, we conclude that
\begin{eqnarray*}
 \mathbb{P}\left\{\frac{W}{k}\geq
x\right\}\leq\left[\left(\frac{k}{px}\right)^x\frac{1}{(1-x)^{1-x}}\right]^k\ .
\end{eqnarray*}
Since $\mathbb{P}(W=k)=[k/p]^k$, this upper bound is also valid when
$x=1$.
\end{proof}

\subsection{Proof of Proposition \ref{prte_prediction_minoration_minimax}}
We derive this minimax lower bound from the hypothesis testing problem
$\{\theta_0=0_p\}$
studied in Section \ref{section_linear_testing}.
Since the covariance $\Sigma=I_p$, the loss
$\mathbb{E}\left[\{X(\theta_1-\theta_2)\}^2/\sigma^2\right]$
is simply $\|\theta_1-\theta_2\|^2_p/\sigma^2$. For the sake of simplicity, we
assume that $p$ is even. We split the $p$ covariates into
two groups $M_1$ and $M_2$ of size $p/2$. Given some $\rho>0$, we fix
$\sigma^2=(1+\rho^2)^{-1}$ and we consider the
two sets 
\begin{eqnarray*}
  \Theta_1[\rho] & = &\Theta[k,p]\cap\left\{\theta: \mathrm{supp}(\theta)\subset
M_1\text{
and }\|\theta\|^2_p=\frac{\rho^2}{1+\rho^2}\right\}\\
\Theta_2[\rho] &= &\Theta[k,p]\cap\left\{\theta: \mathrm{supp}(\theta)\subset
M_2\text{
and }\|\theta\|^2_p=\frac{\rho^2}{1+\rho^2}\right\}\ . 
\end{eqnarray*}
Take any estimator $\widehat{\theta}$. We consider an estimator
$\widetilde{\theta}\in \Theta_1[\rho]\cup \Theta_2[\rho]$  such that 
\begin{eqnarray*}
 \|\widetilde{\theta}-\widehat{\theta}\|_p= \min_{\theta \in \Theta_1[\rho]\cup
\Theta_2[\rho]}\|\theta-\widehat{\theta}\|_p\ .
\end{eqnarray*}
By the triangle inequality, we have
$\|\widetilde{\theta}-\theta_0\|_p\leq 2\|\widehat{\theta}-\theta_0\|_p$, for any
$\theta_0\in \Theta_1[\rho]\cup \Theta_2[\rho]$.
\begin{eqnarray}\label{minoration_minimax_preuve_prediction}
\sup_{i=1,2}\sup_{\theta_0\in\Theta_i[\rho]}
\mathbb{E}_{\theta_0,\sigma}\left[\frac{\|\widehat{\theta}-\theta_0\|_p^2}{\sigma^2}\right]\geq
\frac{\rho^2}{4}\sup_{i=1,2}\sup_{\theta_0\in\Theta_i[\rho]
}\mathbb{P}_{\theta_0,\sigma}[\supp(\widetilde { \theta } )\nsubseteq M_i]\ .
\end{eqnarray}
It is enough to prove that for $\rho^2= C_1\frac{k}{n}\log\left(p\right)
\exp\{C_2\frac{k}{n}\log\left(p\right)\}$,  the supremum of
the probabilities 
$\mathbb{P}_{\theta_0,\sigma}[\supp(\widetilde { \theta } )\nsubseteq M_i]$ is lower
bounded by a positive constant.
This  is equivalent to lower bounding the minimax separation distance for 
${\bf H_0}:$ $\{\theta_0\in \Theta_1[\rho]\text{ and } \sigma^2=(1+\rho^2)^{-1}\}$  against ${\bf H_1}$: $\{\theta_0\in \Theta_2[\rho]\text{ and }\sigma^2=(1+\rho^2)^{-1}\} $.\\

As in the proof of Theorem \ref{thrm_minimax_testing}, we build a prior 
distribution $\mu_{1,\rho}$ on $\theta_0$. Consider the
collection $\mathcal{M}_1(k)$ of subsets of $M_1$ of size $k$. Let $\hat{m}$
be be some random
variable uniformly distributed over $\mathcal{M}_1(k)$.
 Then, 
$\mu_{1,\rho}$ is the distribution  of 
$\widehat{\theta}=
\sum_{j\in\hat{m}} \rho/\sqrt{k(1+\rho^2)}e_j$.
Similarly, we define the prior distribution $\mu_{2,\rho}$ on $\Theta_2[\rho]$.
We note $\mathbb{P}_{\mu_{i,\rho},\sigma}= \int \mathbb{P}_{\theta_0,\sigma}d{\mu_{i,\rho}}$.
We
have
\begin{eqnarray}
\sup_{i=1,2}\sup_{\theta_0\in\Theta_i[r]
}\mathbb{P}_{\theta_0}[\supp(\widetilde { \theta } )\nsubseteq M_i] &\geq &
 1 - \frac{1}{2}\|\mathbb{P}_{\mu_{1,\rho},\sigma}-\mathbb{P}_{\mu_{2,\rho},\sigma}\|_{TV}\ .\nonumber\\
&\geq & 1-  \|\mathbb{P}_{\mu_{1,\rho},\sigma}-\mathbb{P}_{0_p,1}\|_{TV}\
,\label{minoration_test_prediction_preuve}
\end{eqnarray}
by the triangle inequality. Lemma \ref{lemma_le_cam_method} states 
that$$
\|\mathbb{P}_{\mu_{1,\rho}}-\mathbb{P}_{0_p,1}\|_{TV}\leq \mathbb{E}_{0_p,1}\left[L^2_{\mu_{1,\rho}}-1\right]\ ,$$
where $L_{\mu_{1,\rho}}=
\text{d}\mathbb{P}_{\mu_{1,\rho},\sigma}/\text{d}\mathbb{P}_{0_p,1}$.
In fact, the second moment of  $L_{\mu_{1,\rho}}$ has been studied in the proof
of Theorem \ref{thrm_minimax_testing}. If we take $\alpha+\delta=53\%$ in
this proof, we derive 
$$\mathbb{E}_0\left[L^2_{\mu_{1,\rho}}\right]\leq 1.9\ ,\quad\quad\text{
if }\frac{\rho^2}{1-\rho^2}\leq
C_1k\log(p)/n\exp(C_2k\log(p)/n)\,\text{ and if }\, p\geq k^3\vee C_3\ .$$

\noindent
Gathering this result with Equations
(\ref{minoration_minimax_preuve_prediction}) and
(\ref{minoration_test_prediction_preuve}) allows to conclude.

\subsection{Proof of Proposition \ref{prte_adaptation_impossible}}
Let us set
$\alpha=\delta=0.01$.
Consider a design ${\bf X}$ that achieves the bound
(\ref{borne_minoration_test_fixed_unknown}) and take $\rho=\rho^*_{F,U}[k,{\bf
X}]/2$. If $k\log(p)/n$ is large enough, then $\rho\geq \sqrt{2}$.
Take any estimator $\widehat{\theta}$ that does not rely on the variance
$\sigma^2$. Let us build a test $T$ of the
hypotheses 
${\bf H_0}$:\ $\{\theta_0=0\text{ and } \sigma>0\}$ against ~\\${\bf H_1}$:
$\{{\theta_0\in\Theta[k,p]} \text{ and } \sigma>0,\ \|{\bf X}\theta_0\|_n^2/(n\sigma^2)\geq
\rho^2\}$:
\begin{eqnarray*}
 T = \left\{\begin{array}{ccc}
      0&\text{ if }& 2\|{\bf X}\widehat{\theta}\|_n^2<\|{\bf Y}\|_n^2\\
1 & \text{ if }& 2\|{\bf X}\widehat{\theta}\|_n^2\geq \|{\bf Y}\|_n^2
     \end{array}\right.
\end{eqnarray*}
By Proposition \ref{prte_minoration_minimax_test_fixed_unknown_variance}, we
have at least one of the two following properties:
\begin{eqnarray}\label{mauvais_niveau}
\sup_{\sigma>0} \mathbb{P}_{0_p,\sigma}(T=1)\, &\geq& \alpha\\
\sup_{\sigma>0,\ \theta_0\in\Theta[k,p],\ \|{\bf X}\theta_0\|_n^2/(n\sigma^2)\geq
\rho^2} \mathbb{P}_{\theta_0,\sigma}(T=0)&\geq&
\delta\label{mauvaise_puissance}
\end{eqnarray}

\noindent 
{\bf CASE 1}: (\ref{mauvais_niveau}) holds. We have $\mathbb{P}_{0_p,\sigma}[\|{\bf Y}\|_n^2\geq n\sigma^2/2]\geq 1-e^{-n/16}$ for any $\sigma>0$. Thus, there exists
$\sigma>0$ such that $\|{\bf X}\widehat{\theta}\|_n^2\geq n\sigma^2/4$ with
probability larger than
$\alpha/2-e^{-n/16}$. As a consequence, we have 
$$\sup_{\sigma>0}\mathbb{E}_{0_p,\sigma}\left[\|{\bf
X}(\widehat{\theta}-\theta_0)\|_n^2/[n\sigma^2]\right]\geq C\ .$$

\vspace{0.3cm}

\noindent
{\bf CASE 2}: (\ref{mauvaise_puissance}) holds. The random
variable $\|{\bf Y}\|_n^2/\sigma^2$ follows a noncentral $\chi^2$ distribution
with $n$ degrees of freedom  and a non centrality parameter $\|{\bf
X}\theta_0\|_n^2/\sigma^2$. By Lemma 1 in Birg\'e~\cite{birge01}, we have
$\|{\bf Y}\|_n^2\leq 3/2 \left[n\sigma^2+\|{\bf X}\theta_0\|_n^2\right]\ ,$
with probability larger than $1-e^{-Cn}$. Consequently, there exist $\sigma>0$
and $\theta_0\in\Theta[k,p]$ such that $\|{\bf X}\theta_0\|_n^2/(n\sigma^2)\geq
\rho^2$ and
$$\|{\bf X}\widehat{\theta}\|_n^2/(n\sigma^2)\leq
\frac{3}{4}\left[1+\|{\bf
X}\theta_0\|_n^2/(n\sigma^2)\right]\leq \frac{7}{8}\|{\bf
X}\theta_0\|_n^2/(n\sigma^2) ,$$
with probability $\delta/2-e^{-Cn}$, since $\rho^2\geq 2$. Thus, we get
\begin{eqnarray*}
 \mathbb{E}_{\theta_0,\sigma}\left[\|{\bf
X}(\widehat{\theta}-\theta_0)\|_n^2/n\right]&\geq& 
\mathbb{E}_{\theta_0,\sigma}\left[\left(\|{\bf
X}\widehat{\theta}\|_n- \|{\bf
X}\theta_0\|_n \right)^2/n\right]\\ &\geq& C \|{\bf
X}\theta_0\|_n^2/n\ \geq C \rho^2\sigma^2\ .
\end{eqnarray*}

\subsection{Proof of Proposition \ref{prte_minoration_variance}}
For the sake of conciseness, we note 
$l(\widehat{\sigma},\sigma)= |\widehat{\sigma}^2/\sigma^2-\sigma^2/\widehat{\sigma}^2|$. 
Given a positive number $\rho$, we note $\sigma_0=(1+\rho^2)^{-1/2}$. As in the proof of Theorem \ref{thrm_minimax_testing}, we consider the prior probability $\mu_{\rho}$ on $\Theta[k,p]$. For any estimator $\widehat{\sigma}>0$, we define $\widetilde{\sigma}$ by 
$\widetilde{\sigma}\in\arg\min_{\sigma\in\{1,\sigma_0\}}l(\widehat{\sigma},\sigma)$. 
For any $\sigma\in\{1,\sigma_0\}$, the loss   $l(\widehat{\sigma},\sigma)$ is controlled as follows:
$$l(\widehat{\sigma},\sigma)\geq \mathbf{1}_{\widetilde{\sigma}\neq \sigma}l(1,\sqrt{\sigma_0})\ .$$
Thus, we get the minimax lower bound
\begin{eqnarray}
\nonumber\lefteqn{\inf_{\widehat{\sigma}} \sup_{\sigma>0,\ \theta_0\in\Theta[k,p]}\mathbb{E}_{\theta_0,\sigma}\left[l(\widehat{\sigma},\sigma)\right] }&&\\&\geq&  \inf_{\widehat{\sigma}>0}\max\left[\mathbb{E}_{0_p,1}\left\{l(\widehat{\sigma},1)\right\};\mathbb{E}_{\mu_{\rho},\sigma_0}\left\{l(\widehat{\sigma},\sigma_0)\right\}\right] \nonumber\\ & \geq & l(1,\sqrt{\sigma_0})
\inf_{\widetilde{\sigma}\in \{1,\sigma_0\}}\max\left[\mathbb{P}_{0_p,1}[\widetilde{\sigma}\neq 1];\mathbb{P}_{\mu_\rho,\sigma_0}[\widetilde{\sigma}\neq \sigma_0]\right] \nonumber\\
&\geq &\frac{l(1,\sqrt{\sigma_0})}{2}\left[1-\frac{\|\mathbb{P}_{0_p,1}-\mathbb{P}_{\mu_\rho,1}\|_{TV}}{2}\right]\nonumber \\
&\geq &\frac{\rho^2}{2\sqrt{1+\rho^2}}\left[1-\frac{1}{2}\left(\mathbb{E}_{0_p,1}[L_{\mu_{\rho}}^2({\bf Y},{\bf X})]-1\right)^{1/2}\right]\label{eq_principale}\ . 
\end{eqnarray}
Let us note two numbers $\eta_1= 1.5$ and $\eta_2=1.8$. If ${\bf X}$ is a standard  Gaussian design and if $k\leq p^{1/3}$, then  the proof of Theorem \ref{thrm_minimax_testing} states for 
$$\rho^2\leq C_1\frac{k}{n}\log\left(\frac{p}{k}\right)\exp\left[C_2\frac{k}{n}\log\left(\frac{p}{k}\right)\right]\ ,$$ we have
$\mathbb{E}_{0_p,1}[L_{\mu_{\rho}}^2({\bf Y},{\bf X})]\leq 1+\eta_1^2$ where the expectation is taken both with respect to ${\bf Y}$ and ${\bf X}$. Applying Markov's inequality, we derive that  with positive probability,
$$\mathbb{E}_{0_p,1}[L_{\mu_{\rho}}^2({\bf Y},{\bf X})|{\bf X}]\leq 1+\eta_2^2\ . $$
For such designs ${\bf X}$ and such $\rho$ we have
\begin{eqnarray*}
\inf_{\widehat{\sigma}} \sup_{\sigma>0,\ \theta_0\in\Theta[k,p]}\mathbb{E}_{\theta_0,\sigma}\left[l(\widehat{\sigma},\sigma)\right]\geq C\frac{\rho^2}{\sqrt{1+\rho^2}}\geq C'\left(\rho\wedge \rho^2\right)\ ,
\end{eqnarray*}
since $\rho^2/\sqrt{1+\rho^2}\geq (\rho\wedge \rho^2)/\sqrt{2}$. We conclude that
\begin{equation*}
\inf_{\widehat{\sigma}} \sup_{\sigma>0,\ \theta_0\in\Theta[k,p]}\mathbb{E}_{\theta_0,\sigma}\left[l(\widehat{\sigma},\sigma)\right]\geq C'_1\frac{k}{n}\log\left(\frac{p}{k}\right)\exp\left[C'_2\frac{k}{n}\log\left(\frac{p}{k}\right)\right]\  .
\end{equation*}

\subsection{Fano's Lemma}

The next lower bounds are established applying Birg\'e's version of Fano's Lemma~\cite{birgelemma}. More precisely, we shall use the following lemma, which is taken from Corollary 2.19 in~\cite{massartflour},

\begin{lemma}\label{lemme_fano}
Let $(S,d)$ be some pseudo metric space, $\{\mathbb{P}_s,s\in\mathcal{S}\}$ be some statistical model. Let us note $\kappa=2e/(2e+1)$. Then, for any estimator $\widehat{s}$ and any finite subset $\mathcal{C}$ of $\mathcal{S}$, setting $\delta=\min_{s,t\in\mathcal{C},\ s\neq t}d(s,t)$, provided that $\max_{s,t}\mathcal{K}(\mathbb{P}_s,\mathbb{P}_t)\leq \kappa \log|\mathcal{C}|$ the following lower bound holds for any $p\geq 1$:
$$\sup_{s\in\mathcal{C}}\mathbb{E}_s\left[d^p(s,\widehat{s})\right]\geq 2^{-p}\delta^p(1-\kappa)\ .$$
 
\end{lemma}

\subsection{Proof of the lower bounds of Propositions \ref{prte_minoration_inverse_fixe} and \ref{prte_minimax_inverse_random}}

This lower bound is based on Fano's lemma. For the sake of
simplicity, we assume that $2k\leq p$ and that $\sigma^2=1$. First, we consider
a unit vector $\theta\in\Theta[2k,p]$ such that $\|{\bf X}\theta\|_n^2=
\varPhi_{2k,-}({\bf X})$.
Let us define $\kappa=2e/(2e+1)$. It is possible to find two vectors
$(\theta_1,\theta_2)\in\Theta[k,p]$ such that $\theta_1-\theta_2= \theta\sqrt{2\kappa\log(2)/\varPhi_{2k,-}({\bf X})}$ and $\mathrm{supp}(\theta_1)\cap
\mathrm{supp}(\theta_2) =\emptyset$. Consequently, the Kullback distance
$\mathcal{K}(\theta_1,\theta_2)$ between the two distributions
$\mathbb{P}_{\theta_1}$ and $\mathbb{P}_{\theta_2}$ is exactly $\kappa\log(2)$ and $\|\theta_1-\theta_2\|_p^2= 2\kappa\log(2)/\varPhi_{2k,-}({\bf X})$.
Applying Lemma \ref{lemme_fano}, we derive the first
part of the lower bound:
 $$\mathcal{RI}_F[k,{\bf X}]\geq C
\frac{1}{\varPhi_{2k\wedge p,-}\left({\bf X}\right)}\ .$$

~\\ Let us turn to the second part of the lower bound. We consider $\mathcal{M}(k,p)$ the
collections of subsets of $\{1,\ldots ,p\}$ of size $k$. 
Applying combinatorial results such as Varshamov's lemma and Lemma 4.10 in \cite{massartflour}, we derive that there exists $\mathcal{M}'(k,p)\subset \mathcal{M}(k,p)$ of size larger than $\exp[Ck\log(ep/k)]$ such that any pairs of distinct sets $m_1$, $m_2$ in $\mathcal{M}'(k,p)$, we have $|m_1\cap m_3|\leq 3k/4$.

For any $m\in\mathcal{M}'(k,p)$, we
define a vector $\theta_{m}$ that satisfies:
\begin{itemize}
 \item  $|(\theta_{m})_i|=1/\sqrt{k}$
if $i\in m$ and $0$ else. 
\item $\|{\bf X}\theta_m\|_n^2\leq \varPhi_{1,+}({\bf X})$. 
\end{itemize}
Let us prove that this construction is possible by induction on $k$. The construction is  straightforward for $k=1$. Assume that this construction is possible for $k-1$. Let us take some subset $m\in\mathcal{M}(k,p)$ and $m'\subset m$ such that $|m'|=k-1$. There exists a vector $\theta$ such that $\textrm{supp}(\theta)=m'$, $|(\theta)_i|=1/\sqrt{k}$ for any $i\in m'$ and $\|{\bf X}\theta\|_n^2\leq \varPhi_{1,+}({\bf X})(k-1)/k$. Now consider the two vectors $\theta_1$ and $\theta_2$ such that $(\theta_1)_i=(\theta_2)_i=\theta_i$ if $i\in m'$, $(\theta_1)_i=-(\theta_2)_i=1/\sqrt{k}$ if $i\in m\setminus m'$ and $(\theta_1)_i=-(\theta_2)_i=0$ else. It follows that 
$\|{\bf X}\theta_1\|_n^2\leq \varPhi_{1,+}({\bf X})$ or $\|{\bf X}\theta_2\|_n^2\leq \varPhi_{1,+}({\bf X})$, which allows to conclude.

For any $r>0$, we consider the set $\mathcal{C}'_k[r]:=\{r\theta_m\ ,\ m\in \mathcal{M}'(k,p)\}$. The Kullback distance between any two element $\theta_1\neq\theta_2$ in $\mathcal{C}'_k[r]$ is upper bounded as follows:
\begin{eqnarray*}
 \mathcal{K}(\theta_1,\theta_2) = \frac{\|{\bf
X}(\theta_1-\theta_2)\|_n^2}{2\sigma^2}\leq 2\varPhi_{1k,+}({\bf X})\frac{r^2}{\sigma^2}\ ,
\end{eqnarray*}
while we have $\|\theta_1-\theta_2\|_p^2\geq r^2/2$.
Applying Birg\'e's version of Fano's lemma~\cite{birgelemma} we conclude that:
\begin{eqnarray*}
\inf_{\widehat{\theta}}\sup_{\theta_0\in
\mathrm{Conv}[\mathcal{C}_k^p(\sqrt{k}r)]}
\mathbb{E}_{\theta_0,\sigma}\left[\|{\bf X}(\widehat{\theta}-\theta_0)\|_n^2/n\right]	\geq
C\left[r^2\wedge
\frac{k(1+\log(p/k))}{\varPhi_{1k,+}({\bf X})}\sigma^2\right]\ ,
\end{eqnarray*}
where $\mathrm{Conv}[A]$ stands for the convex hull of $A$. Taking
$r^2=k[1+\log(p/k)]\sigma^2/\varPhi_{2k,+}({\bf X})$ allows to conclude.

The proof of the minimax lower bound (\ref{minoration_inverse_design_random_rask}) in Proposition \ref{prte_minimax_inverse_random} follows exactly the same steps. The minimax lower bound (\ref{eq_risque_minimax_inverse_random}) is a consequence of (\ref{minoration_inverse_design_random_rask}) and the fact that $\varPhi_{1,+}(\sqrt{\Sigma})=1$ for any $\Sigma\in \mathcal{S}_p$.

\subsection{Proof of Proposition
\ref{cor_minoration_inverse_dimension_pas_trop_grande}}

{\bf Proof of the first result}.
First, the minimax lower bound is a straightforward consequence of (\ref{minoration_inverse_design_fixe}), since $\varPhi_{1,+}({\bf X})=n$ if ${\bf X}\in \mathcal{D}_{n,p}$. Let us turn to the upper bound. Thanks to the minimax upper bound (\ref{minoration_inverse_design_fixe_rask}), we only have to prove that there exists a design ${\bf X}$ such that its $2k$-restricted eigenvalues remain close from $n$.

Consider a standard Gaussian design ${\bf W}$ of size $n\times p$. 
Rescaling to a norm of  $\sqrt{n}$ each column of ${\bf W}$, we get a design ${\bf X}\in\mathcal{D}_{n,p}$.
Let us  assume that $k[1+\log(p/k)]\leq \{4(1+\sqrt{2})\}^{-2}n$.
Applying Lemma \ref{lemma_concentration_vp_wishart}, we control the restricted
eigenvalues of ${\bf W}$:
\begin{eqnarray*}
\varPhi_{2k,+}({\bf W}/\sqrt{n})\leq (7/4)^2\text{ and }
 \varPhi_{2k,-}({\bf W}/\sqrt{n})\geq (1/4)^2\ ,
\end{eqnarray*}
with probability larger than $1-\exp(-n/32)$.
Consider any $\theta\in\Theta[2k,p]$ such that $\|\theta\|_p=1$. By definition
of ${\bf X}$, there exists some $\theta'\in\Theta[2k,p]$ such that ${\bf
X}\theta={\bf W}\theta'$. Moreover we have $$\|\theta'\|_p^2\geq \varPhi_{1,+}^{-1}({\bf
W}/\sqrt{n})\ .$$ Hence, we derive that
$$\varPhi_{2k,-}({\bf X})\geq
\varPhi_{2k,-}({\bf
W})\varPhi_{1,+}^{-1}({\bf
W}/\sqrt{n})\ .$$
Thus, we have  $\varPhi_{2k,-}({\bf X})\geq n/49$ with positive probability.\\

\noindent 
{\bf Proof of the second result}.
Let ${\bf X}$ be a design in $\mathcal{D}_{n,p}$.
Take $\delta\in (0,1]$.  Let us consider the collection $\mathcal{M}(k,p)$ (defined in Section \ref{section_notation}). As explained in the proof of  Proposition \ref{prte_minoration_inverse_fixe}, there exists $\mathcal{M}'(k,p)\subset \mathcal{M}(k,p)$ of size larger than $\exp[Ck\log(ep/k)]$ such that any pairs of distinct sets $m_1$, $m_2$ in $\mathcal{M}'(k,p)$, we have $|m_1\cap m_3|\leq 3k/4$.

For any
$m\in\mathcal{M}'(k,p)$, we
define a vector $\theta_{m}$ such that  $|(\theta_{m})_i|=1/\sqrt{k}$
if $i\in m$ and $0$ else and that $\|{\bf X}\theta_m\|_n^2\leq n$. Such a construction is justified in the proof of Proposition \ref{prte_minoration_inverse_fixe}.

For any $m_1\neq m_2$ in $\mathcal{M}'(k,p)$, we have
$\|\theta_{m_1}-\theta_{m_2}\|_p^2\geq 1/2$. If there exist two distinct sets $(m_1,m_2)\in
\mathcal{M}'(k,p)$ such that $\|{\bf X}(\theta_{m_1}-\theta_{m_2})\|_n^2\leq
n\delta^2$, then the design ${\bf X}$ satisfies
$\varPhi_{2k,-}({\bf X})\leq 2n\delta^2$. A necessary condition for ${\bf X}$
to satisfy
$\varPhi_{2k,-}({\bf X})\geq 2n\delta^2$ is therefore that the vectors ${\bf
X}\theta_m$ are $\sqrt{n}\delta$-separated.  \\

If ${\bf X}$ satisfies
$\varPhi_{2k,-}({\bf X})\geq 2n\delta^2$, then the balls in $\mathbb{R}^n$ with radius
$\sqrt{n}\delta$ centered at ${\bf X}\theta_m$  are all disjoint. Thus, the sum of their volumes, 
 is smaller than the volume of a ball a radius $\sqrt{n}(1+\delta)$ in $\mathbb{R}^n$. This implies that
$\delta\leq
2(k/ep)^{Ck/n}$. Hence, for any design ${\bf X}$ with unit
columns,
we have
$$\varPhi_{2k,-}({\bf X})\leq C_1\left(\frac{k}{ep}\right)^{C_2k/n}\ ,$$
which allows to prove the second result.\\

\noindent 
{\bf Proof of the third result}. The minimax lower bound is direct consequence of (\ref{minoration_inverse_design_fixe}) and (\ref{minoration_valeur_propre_restreinte_design_fixe}). In order to finish the proof, we shall combine the minimax upper bound (\ref{minoration_inverse_design_fixe_rask}) with an upper bound of 
$\inf_{{\bf
X}\in\mathcal{D}_{n,p}} \varPhi^{-1}_{2k,-}({\bf X})$. Consider a
standard Gaussian design ${\bf X}$ with size $n\times p$. Applying the
deviation inequality (\ref{majoration_wishart_queue_distribution}) of Lemma
\ref{lemma_concentration_vp_wishart}, we derive that with probability larger than $1-1/p$, we have
$$\varPhi^{-1}_{2k,-}({\bf X})\leq n
C_1\left(\frac{p}{k}\right)^{C_2k/n}\left[\frac{k}{n}
\log\left(\frac{p}{k}\right)\vee 1\right]\ .$$
However, the design ${\bf X}$ does not belong to
$\mathcal{D}_{n,p}$.  This is why we consider ${\bf X'}= {\bf X}D^{-1}$,
where $D$ is a diagonal matrix of size $p$, whose $l$-th diagonal element
corresponds to the norm of the $l$-th column of ${\bf X}/\sqrt{n}$.
Obviously, 
${\bf X'}$ belongs to $\mathcal{D}_{n,p}$.
\begin{eqnarray*}
 \varPhi_{2k,-}({\bf X}')=\inf_{\theta\in\Theta[k,p]}\frac{\|{\bf
X'}\theta\|_n^2}{\|\theta\|_{p}^2}=\inf_{\theta\in\Theta[k,p]}\frac{\|{
\bf
X}\theta\|_n^2}{\|D\theta\|_{p}^2} \geq 
\frac{\varPhi_{2k,-}({\bf
X})}{\varphi_{\max}(D^2)}\ ,
\end{eqnarray*}
Each diagonal element of $nD^2$ follows of $\chi^2$ distribution with $n$
degrees of freedom. Applying Lemma \ref{lemma_concentration_chi2}, we derive that
$\varphi_{\max}(D)\leq C\sqrt{1\vee  \log(p)/n}$ with probability larger than $1-1/p$. We conclude that
\begin{eqnarray*}
\varPhi^{-1}_{2k,-}({\bf X}')&\leq&
C_1n\left(\frac{p}{k}\right)^{C_2k/n)}\left[\frac{k}{n}
\log\left(\frac{p}{k}\right)\vee 1 \right]\left[1\vee \frac{\log(p)}{n}\right]\\
 &\leq & C'_1n\left(\frac{p}{k}\right)^{C'_2k/n)}\ .
\end{eqnarray*}
with probability larger than $1-2/p$. This allows to conclude.

\subsection{Proof of Proposition
\ref{prte_minoration_support_estimation}}
For the sake of simplicity, we assume that $\sigma^2=1$.
Consider a design ${\bf X}\in\mathcal{D}_{n,p}$.
By the proof of Proposition \ref{cor_minoration_inverse_dimension_pas_trop_grande}, there exist
two vectors $\theta_1$ and $\theta_2$ such that:
\begin{enumerate}
 \item  $\theta_1$ and
$\theta_2$
contain exactly $k$ non-zero components which are all equal to $1/\sqrt{k}$ in absolute value. 
\item The Hamming distance between $\theta_1$ and $\theta_2$ is larger than $k/2$.
\item  $\|{\bf
X}(\theta_1-\theta_2)\|_n^2\leq C_1n\exp\left[-C_2k/n\log(ep/k)\right]:=\rho^{*-2}$.
\end{enumerate}

Let us set $\theta^*_1=C\rho^*\theta_1$ and $\theta^*_2= C\rho^*\theta_2$ with
$C=4\log(2)e/(2e+1)$. Consequently, the Kullback discrepancy between
$\mathbb{P}_{\theta^*_1}$ and $\mathbb{P}_{\theta^*_2}$ is smaller than $\log(2)2e/(2e+1)$.
Consider an estimator $\widehat{\theta}$ taking its values in
$\{\theta_1^*,\theta_2^*\}$.
Applying Corollary 2.18 in \cite{massartflour} (which is another version of Fano's Lemma), we derive that
$\inf_{\theta_0\in\{\theta_1^*,\theta_2^*\}} \mathbb{P}_{\theta_0,1}(\widehat{\theta}=\theta_0)\leq
2e/(2e+1)$. This allows to conclude.

\subsection{Proof of Proposition \ref{prte_minoratoin_subset_estimation}}
For the sake of simplicity, we assume that $\sigma^2=1$ and that $p$
is even.
Consider any estimator $\widehat{M}$ of size $p_0$. We set 
\begin{equation}\label{eq_rho}
\rho^2=
\frac{C_1}{2}\frac{k}{n}\log(p)\exp\left[\frac{C_2}{2}\frac{k}{n}\log(p)\right]
\end{equation}
 where the constants $C_1$, $C_2$ correspond to
the ones used at the end of the proof of
Proposition \ref{prte_prediction_minoration_minimax}. We also consider the set
$\mathcal{C}_k^p(\rho)$. Suppose
that we have
\begin{eqnarray}\label{hypothese_controle_support}
\sup_{\theta_0 \in
\mathcal{C}_k^p(\rho)}\mathbb{P}_{\theta_0,1}[\mathrm{supp}(\theta_0)\subset
\widehat{M}]\geq
7/8\ .
\end{eqnarray}

Assume we are given a second $n$-sample of ($Y,X$) independent of the first one.
We note (${\bf Y}',{\bf X}'$) this new sample. We consider the estimator
$\widetilde{\theta}_k$ defined by
$$\widetilde{\theta}_k:=\arg\min_{\theta\in\Theta[k,p]\text{ and
}\mathrm{supp}(\theta)\subset \widehat{M}}\|{\bf Y}'-{\bf
X}'\theta\|_n^2\ .$$
Since $\Sigma=I_p$, all the covariates that do
not lie in the support of $\theta_0$ play a symmetric role in the distribution of
(${\bf Y},{\bf X}$). This estimator
$\tilde{\theta}_k$ has the same form as the estimator $\widehat{\theta}_k$
introduced in (\ref{definition_moindre_carre_sparse}).
Arguing as in the proof of  Theorem \ref{thrm_prediction_majoration_minimax}, we
derive that 
\begin{eqnarray*}
 \|\widetilde{\theta}_k-\theta_0\|_p^2\mathbf{1}_{\mathrm{supp}(\theta_0)\subset
\widehat{M}}\leq C'_1 k\log\left(\frac{ep_0}{k}\right)
\exp\left[C'_2\frac{k}{n}\log\left(\frac{
ep_0} { k } \right)\right ] \ ,
\end{eqnarray*}
with probability larger than $7/8$. Gathering this bound with
(\ref{hypothese_controle_support}), we derive that for any $\theta_0\in
\mathcal{C}_k^p(\rho)$, we have
\begin{eqnarray}\label{controle_estimateur_tres_grande_dimension}
 \|\widehat{\theta}_k-\theta_0\|_p^2\leq C'_1
\frac{k}{n}\log\left(\frac{ep_0}{k}\right)\exp\left[C'_2\frac{k}{n}\log\left(\frac{
ep_0} { k } \right)\right ] ,
\end{eqnarray}
with probability larger than $3/4$.\\

We shall prove that (\ref{controle_estimateur_tres_grande_dimension}) is
impossible if $p_0$ is too large.
Let us  split the $p$ covariates into two groups $M_1$ and $M_2$. We
consider the subsets $\mathcal{C}_{k,1}^p(\rho)$ (resp. 
$\mathcal{C}_{k,2}^p(\rho)$)  of $\mathcal{C}_{k}^p(\rho)$
whose elements have their support in $M_1$ (resp. $M_2$). Arguing as in
(\ref{minoration_minimax_preuve_prediction}) and
(\ref{minoration_test_prediction_preuve}), we derive that for any
estimator $\widehat{\theta}$, there exists $\theta_0\in
\mathcal{C}_{k,1}^p(\rho)\cup \mathcal{C}_{k,2}^p(\rho)$ such that 
\begin{eqnarray*}
 \|\widehat{\theta}-\theta_0\|_p^2\geq \frac{\rho^2}{4}=\frac{C_1}{8}\frac{k}{n}\log(p)\exp\left[\frac{C_2}{2}\frac{k}{n}\log(p)\right]\ ,
\end{eqnarray*}
with probability larger than $1/4$. Here, the constants $C_1$ and $C_2$ are the same as in (\ref{eq_rho}).  \\

The last lower bound
contradicts (\ref{controle_estimateur_tres_grande_dimension}) is
$\log(p_0)/\log(p)\leq \delta$, where $\delta>0$ depends on the relative values of $C_1$, $C_2$, $C'_1$, and $C'_2$ in (\ref{eq_rho}) and (\ref{controle_estimateur_tres_grande_dimension}).

\section{Procedures involved in the proofs of the minimax upper bounds}\label{section_proof_upperbound}

\subsection{Testing procedures}

\subsubsection{Known variance: test  $T^*_{\alpha}$}\label{section_testetoile}

In order to establish the minimax upper bounds for known variance, we consider the
following testing procedure. It is taken from Baraud~\cite{baraudminimax} who
applies it in the Gaussian sequence model. In the sequel, $\bar{\chi}_k(u)$ denotes the probability for a $\chi^2$ distribution with $k$ degrees of freedom to
be larger than $u$. Given a subset $m$ of $\{1,\ldots ,p\}$, $ \Pi_{m}$ refers to the
orthogonal projection onto the space generated
by the vectors $({\bf X}_i)_{i\in m}$.

\begin{defi}{\bf [Procedure
$T^*_{\alpha}$]}\label{definition_test_variance_connue}
 Define $k^*$ as the smallest integer such that
$k^*[1+\log(p/k^*)]\geq
\sqrt{n}$.  For any $1\leq
k<k^*$, we define the statistics $T^*_{\alpha,k}$  by
\begin{eqnarray*}
 T^*_{\alpha,k}:= \sup_{m\in\mathcal{M}(k,p)}\|\Pi_m{\bf Y}\|_n^2-
\sigma^2 \bar{\chi}^{-1}_k\big[\alpha/{\scriptstyle \binom{p}{k}}\big]\ ,
\end{eqnarray*}
where $\mathcal{M}(k,p)$ is defined in Section \ref{section_notation}.  We also consider $$T^*_{\alpha,n}:=
\|{\bf Y}\|_n^2- \sigma^2 \bar{\chi}^{-1}_n(\alpha)\ .$$ The procedure
$T^*_{\alpha}$ is defined by
\begin{eqnarray}\label{definition_test_fixed_design_variance_connue}
 T^*_{\alpha}=\left[\vee_{1\leq k<
k^*}T^*_{\alpha/(2k^*),k}\right]\vee T^*_{\alpha/2,n}\ .
\end{eqnarray}
The hypothesis ${\bf H_0}$ is rejected if $T^*_{\alpha}$ is positive. 
\end{defi}
$T^*_{\alpha,k}$  corresponds to a Bonferroni multiple testing procedure
based on a large number of parametric tests of the hypothesis ${\bf H_0}$:
$\{\theta_0=0_p\}$ against ${\bf H_{1,m}}$: $\{\theta_0\neq 0$ and
$\mathrm{supp}(\theta_0)\subset m\}$ for any $m\in\mathcal{M}(k,p)$. As a consequence, $T^*_{\alpha,k}$ allows to test the hypothesis 
${\bf H_0}$:$\{\theta_0=0\}$ against ${\bf H_{1,k}}$: $\{\theta_0\in\Theta[k,p]\setminus\{0_p\}\}$. Then, $T^*_{\alpha}$ corresponds to a 
Bonferroni multiple testing procedures based on the statistics $T^*_{\alpha,k}$, $k\in\{1,\ldots k^*\}\cup \{n\}$. Obviously, the procedure $T_{\alpha}^*$ is computationally intensive. It is used here  as a theoretical tool to derive minimax upper bounds.

\subsubsection{Unknown variance: test $T_{\alpha}$}\label{section_test}

We introduce a second testing procedure to handle the case of unknown
variance $\sigma^2$.

\begin{defi}{\bf [Procedure
$T_{\alpha}$]}\label{definition_definition_test_unknown_variance}
Fixing some subset $m$ of $\{1,\ldots, p\}$ such that $n-|m|>0$,
we note  $d_m({\bf X})$ the rank of the subdesign ${\bf X}_m$ of ${\bf X}$ of
size $n\times |m|$. We define the Fisher statistic $\phi_m$ by
\begin{eqnarray}\label{definition_phi} 
\phi_m({\bf Y},{\bf X}) := \frac{[n-d_m({\bf X})]\|\Pi_{m}{\bf Y} 
  \|_n^2}{d_m({\bf X})\|{\bf Y} - \Pi_{m}{\bf Y}\|^2_n}\ .
\end{eqnarray}
 We build the  statistic $T_{\alpha,k}({\bf Y},{\bf X})$ as 
\begin{eqnarray}\label{definition_test_general} 
T_{\alpha,k}:= \sup_{m\in \mathcal{M}(k,p)}\phi_m({\bf 
  Y},{\bf X}) - \bar{F}^{-1}_{d_m({\bf X}),n-d_m({\bf
X})}\big[\alpha/{\scriptstyle \binom{p}{k}}\big]\ ,
\end{eqnarray} 
 where $\bar{F}_{k,n-k}(u)$ denotes the probability for a Fisher variable with
$k$ and $n-k$ degrees of freedom to be larger than $u$.
Finally, the statistic  $T_{\alpha}$ is defined by
\begin{eqnarray}\label{definition_test_variance_inconnue}
T_{\alpha}:=\sup_{k=1,\ldots ,\lfloor n/2\rfloor}T_{\alpha/\lfloor
n/2\rfloor,k}\ .
\end{eqnarray}
The hypothesis ${\bf H_0}$ is rejected when $T_{\alpha}$ is positive.
\end{defi}

In fact, $T_{\alpha}$ is  a Bonferroni multiple testing
procedure. Contrary to $T^*_{\alpha}$, it is based on Fisher statistics to handle the
unknown variance. The
ideas underlying this statistic 
 have been introduced in Baraud et al.~\cite{baraud03} in the context of fixed design
regression.

\subsection{Estimation procedures}

\subsubsection{Definition of the estimator $\widetilde{\theta}^V$}\label{section_theta_V}

\begin{defi}{\bf
[Estimator
$\widetilde{\theta}^V$]}\label{definition_estimateur_prediction_random_design}
For any integer $k\in \{1,\ldots,p\}$,  we consider a least-squares  estimator
$\widehat{\theta}_k$ defined by
\begin{eqnarray}\label{definition_moindre_carre_sparse}
\widehat{\theta}_k  \in \arg\min_{\theta \in\Theta[k,p]}\|{\bf Y}-{\bf
X}\theta\|_n^2\ .
\end{eqnarray}
Let us define the penalty function $\pen:\ \{1,\ldots, \lfloor
(n-1)/4\rfloor\} \mapsto \mathbb{R}^+$ by
\begin{eqnarray}\label{penalite}
 \pen(k)= K\frac{k}{n}\log\left(\frac{ep}{k}\right)\ ,
\end{eqnarray}
where $K>0$ is a tuning parameter. The dimension 
$\widehat{k}^{V}$ is selected as follows
\begin{eqnarray*}
 \widehat{k}^V\in \arg\min_{1\leq k\leq \lfloor(n-1)/4\rfloor }\log\left[\|{\bf
Y}-{\bf
X}\widehat{\theta}_k\|_n^2\right] +\pen(k)\ .
\end{eqnarray*}
For short, we note $\widetilde{\theta}^V=\widehat{\theta}_{\widehat{k}^V}$. 
\end{defi}

This variable selection procedure relies on complexity penalization. The
penalty
$\pen(k)$ depends on the size of $k$ and on the number $\binom{p}{k}$ of
subsets of $\{1,\ldots,p\}$ of size $k$. Observe that the estimator
$\widetilde{\theta}^V$ does
not require the knowledge of $\sigma^2$.

The choice of the tuning parameter $K$ is universal: it neither depends on $n$,
$p$, $k$, nor on $\Sigma$, $\theta_0$, $\sigma^2$. It is only constrained to be larger than a positive numerical constant so that the equations (B.8), (B.24), (B.26), (B.31), and (B.34) in the proofs of Theorem \ref{thrm_prediction_majoration_minimax}, Propositions \ref{prte_risque_baraudgiraud} and \ref{prte_majoration_probleme_inverse_adaptation} in~\cite{technical} hold.

\subsubsection{Definition of the estimator $\widetilde{\theta}^{BM}$ and proof of (\ref{eq_adaptation}) in Proposition \ref{prte_minimax_prediction_fixed_design}}\label{section_theta_BM}
\begin{defi}{\bf [Procedure for
fixed design regression]}\label{definition_birge_massart}
Define $k^*$ as the smallest integer $k$ such that $k[1+\log(p/k)]\geq
n$. Let us consider the collection of dimensions
$\mathcal{K}:=\{1,\ldots,k^*\}\cup\{n\}$. Then, 
 the penalty function
$\textrm{pen}:\mathcal{K}\mapsto \mathbb{R}^+$ is defined by
\begin{eqnarray*}
\textrm{pen}(k) :=\left\{ \begin{array}{ccc}                  
4k\left[4 +\log\left(\frac{p}{k}\right)\right] & \text{if} &
k\leq k^*\\
2n &\text{if} & k=n\ ,
                  \end{array}\right.
\end{eqnarray*}
We recall that for $k\leq k^*$, the estimators $\widehat{\theta}_k$ are defined  in (\ref{definition_moindre_carre_sparse}) and that $\widehat{\theta}_n\in\arg\min_{\theta\in\mathbb{R}^p}\|{\bf Y}-{\bf X}\theta\|_n^2$.
The size  $\widehat{k}^{BM}$ is selected by minimizing the following penalized
criterion
\begin{eqnarray}\label{critere_birge_massart}
\widehat{k}^{BM}:=  \arg \inf_{k\in\{1,\ldots k^*\}\cup\{n\}}\|{\bf Y}-{\bf
X}\widehat{\theta}_k\|_n^2+\sigma^2\textrm{pen}(k)\ ,
\end{eqnarray}
For short, we write
$\widetilde{\theta}^{BM}=\widehat{\theta}_{\widehat{k}^{BM}}$.
\end{defi}

Observe that the estimator $\widetilde{\theta}^{BM}$ requires the knowledge of
the variance $\sigma^2$. Then, Eq. (\ref{eq_adaptation}) is a special case of Theorem 1 in Birg\'e and
Massart~\cite{massart_pente}.

\section{Deviation inequalities }

The proofs of the deviation inequalities stated in this section are postponed to Appendix C in~\cite{technical}.

\begin{lemma}[$\chi^2$ distributions]\label{lemma_concentration_chi2}
  For any integer $d>0$ and any number $0<x<1$, 
 \begin{eqnarray*}
\mathbb{P}\left(\chi^2(d) \geq d + 2\sqrt{d\log(1/x)} + 2\log(1/x) \right) &\leq
&x\ 
\nonumber\ ,\\
\mathbb{P}\left(\chi^2(d) \leq d - 2\sqrt{d\log(1/x)} \right)& \leq& x
\nonumber\ .
\end{eqnarray*}
For  any positive number $0<x<1$
\begin{eqnarray}
\mathbb{P}\left[\chi^2(d) \leq
dCx^{2/d}\right] &\leq& x 
\ ,\label{minoration_chi2_petite_proba}
\end{eqnarray}
where the constant $C=\exp(-1)$.
\end{lemma}

\begin{lemma}[Wishart distributions]\label{lemma_concentration_vp_wishart}
Let $Z^TZ$ be a standard Wishart matrix of parameters $(n,d)$ with $n>d$. For
any number $0<x<1$, 
\begin{eqnarray}
\mathbb{P}\left[\varphi_{\max}\left( Z^TZ\right) \geq
n\left(1+\sqrt{d/n}+\sqrt{2\log(1/x)/n}\right)^2 \right]& \leq &x\ 
,\nonumber	\\
\label{majoration_wishart_sous_gaussienne}
\mathbb{P}\left[\varphi_{\min}\left(Z^TZ\right) \leq
n\left(1-\sqrt{d/n}-\sqrt{2\log(1/x)/n}\right)_+^2 \right]& \leq
&x\ .
\end{eqnarray}
For any $(n,d)$ with $n\geq 4d+1$ and any number  $0<x<1$, 
\begin{eqnarray}
\mathbb{P}\left[\varphi_{\min}\left(Z^TZ\right) \leq
nCx^{\frac{2}{n-2d}}\left[1\vee \frac{\log(2/x)}{n}\right]^{-1} \right]& \leq&
x\ ,
\label{majoration_wishart_queue_distribution}
\end{eqnarray}
where $C$ is a numerical constant.
\end{lemma}

The two first deviation inequalities are taken from Theorem 2.13 in
\cite{davidson2001}.  The bound (\ref{majoration_wishart_queue_distribution})
allows to control the tail distribution of the smallest eigenvalue of a
Wishart distribution. Rudelson and
Vershynin~\cite{rudelson} have 
provided a control similar to (\ref{majoration_wishart_queue_distribution})
under subgaussian assumptions. However, their results only holds for events of
probability smaller than $1-e^{-n}$.


\section*{Acknowledgements}
I am grateful to Yannick Baraud and Christophe Giraud, the Associate Editor, and two anonymous referees for 
suggestions that greatly improve the presentation of the paper.

\bibliographystyle{acmtrans-ims}

\bibliography{estimation}

\end{document}